\documentclass[preprint,12pt,number]{elsarticle}

\usepackage{amssymb}
\usepackage{amsmath}
\usepackage{amsthm}

\usepackage{hyperref}

\usepackage{subcaption}

\newtheorem{theorem}{Theorem}
\newtheorem{proposition}[theorem]{Proposition}%
\newtheorem {lemma} [theorem] {Lemma}
\newtheorem{remark}[theorem]{Remark}%
\newtheorem{corollary}[theorem]{Corollary}%

\DeclareMathOperator{\argmin}{arg\,min}

\newcommand{\1}{\mathbf{1}}

\begin{document}

\begin{frontmatter}



\title{Extended Argmin-Theorems for multiple nets of multivariate c\`{a}dl\`{a}g stochastic processes} 


\author[label1]{Dietmar Ferger}
\ead{dietmar.ferger@tu-dresden.de}
\author[label2]{Niklas Rosar}
\ead{niklasaaron.rosar@tu-dresden.de} 

\affiliation[label1]{organization={Institut für Mathematische Stochastik, Technische Universität
Dresden},
            addressline={Zellescher Weg 12-14},
            city={Dresden},
            postcode={01069},
            country={Germany}}

\affiliation[label2]{organization={Formerly at Institut für Mathematische Stochastik, Technische Universität
Dresden},
            addressline={Zellescher Weg 12-14},
            city={Dresden},
            postcode={01069},
            country={Germany}}

\begin{abstract}
Consider finitely many nets of multivariate c\`{a}dl\`{a}g stochastic processes. We show that the vectors consisting of the respective
minimizing points converge in distribution to a random closed set. This set is given as a cartesian product with factors which are equal to the
set of all minimizing points of stochastic processes occurring as functional limits of the respective nets. If these limit processes have almost surely exactly one minimizer, then the vectors converge classically in distribution to the vector of these minimizers.
\end{abstract}




\end{frontmatter}



\section{Introduction} \label{SectionIntroduction}
Assume that for each $n \in \mathbb{N}$ the random map $\xi_n$ is a minimizing point of a stochastic process $X_n$ defined on a probability space $(\Omega, \mathcal{A},\mathbb{P})$ with trajectories in a
function space $\textbf{F}$ endowed with some topology $\mathcal{T}$. If $(X_n)_{n \in \mathbb{N}}$ converges in distribution to a limit process $X$, then the question arises what can be said about a possible distributional convergence of the sequence $(\xi_n)_{n \in \mathbb{N}}$ and how does the limit variable looks like? Argmin-theorems give answers to that question. Since $\xi_n$ is a maximizing point of $-X_n$, every Argmin-theorem yields an Argmax-theorem and vice versa. The first Argmin-theorem in the above spirit goes back to Kim and Pollard (1990) \citep{kim1990cube}. They deal with $\textbf{F}=l^\infty(H)$, the space of all locally bounded real-valued functions on $H=\mathbb{R}^d$ equipped with the topology $\mathcal{T}$ of uniform convergence on compacta. Here, the problem occurs that the stochastic processes $X_n$ need not to be measurable maps into $l^\infty(H)$ endowed with the Borel-$\sigma$ algebra $\sigma(\mathcal{T})$.
As a way out, Kim and Pollard \cite{kim1990cube} use the general concept of distributional convergence in the sense of Hoffmann-J{\o}rgensen (1998) \cite{hoffmann1998convergence}, where no measurability is required. This type of convergence is denoted by
\begin{equation} \label{Hoffmann}
X_n \rightsquigarrow X  \; \text{ in } l^\infty(H).
\end{equation}
It is (for general metric spaces $H$) equivalent to, e.g.
\begin{equation} \label{charopen}
 \liminf_{n \rightarrow \infty}\mathbb{P}_*(X_n \in O) \ge \mathbb{P}(X \in O) \quad \text{for all open } O \subseteq l^\infty(H),
\end{equation}
where $\mathbb{P}_*$ is the inner probability of $\mathbb{P}$. See Theorem 1.3.4 in van der Vaart and Wellner \cite{van1996weak} for several other equivalent characterisations of \eqref{Hoffmann}.
From \eqref{charopen}, one can immediately see how the Hoffmann-J{\o}rgensen-convergence $\rightsquigarrow$ depends on the underlying topology $\mathcal{T}$ on $l^\infty(H)$. Therefore, we occasionally write more precisely $X_n \rightsquigarrow X  \; \text{ in } (l^\infty(H), \mathcal{T})$.
Suppose that in addition to $(\ref{Hoffmann})$, the sequence $(\xi_n)_{n \in \mathbb{N}}$ is stochastically bounded, i.e
\begin{equation} \label{stochasticallybounded}
 \lim_{k \rightarrow \infty} \limsup_{n \rightarrow \infty} \mathbb{P}^*(||\xi_n||>k)=0,
\end{equation}
where $||\cdot||$ denotes any norm on the euclidian space $H=\mathbb{R}^d$ and $\mathbb{P}^*$ is the outer probability of $\mathbb{P}$.
If moreover  $X$ possesses an unique minimizing point $\xi$ with probability one, then Kim and Pollard \cite{kim1990cube} prove that
\begin{equation} \label{Hoffmannconvergence}
\xi_n \rightsquigarrow \xi \; \text{ in } H.
\end{equation}
van der Vaart and Wellner (1996) \cite{van1996weak} extend this result from $\mathbb{R}^d$ to a metric space $H$. In their proof, which is much simpler than that of Kim and Pollard \cite{kim1990cube}, they use that (\ref{Hoffmann}) is equivalent to the convergence of the restrictions $X_n \rightsquigarrow X  \; \text{ in } l^\infty(K)$ for all compact $K \subseteq H$, confer Theorem 1.6.1 in \cite{van1996weak}. However, they have to pay a small price for the more general set up, because the assumption (\ref{stochasticallybounded}) must be replaced by uniform tightness of $(\xi_n)_{n \in \mathbb{N}}$, i.e. for each $\eta>0$ there exists a compact  $K \subseteq H$ such that
\begin{equation} \label{uniformtightness}
\limsup_{n \rightarrow \infty}\mathbb{P}^*(\xi_n \notin K) \le \eta.
\end{equation}
Indeed, in metric spaces $H$ closed and bounded sets are in general not compact, so that the requirement (\ref{uniformtightness}) is stronger than the generalisation of (\ref{stochasticallybounded}) to metric spaces. (Of course, when $H=\mathbb{R}^d$ they are equivalent.)
It should be mentioned that \cite{kim1990cube} and \cite{van1996weak} both consider more generally $\xi_n$, which are $\epsilon_n$-optimal solutions, i.e. $X_n(\xi_n) \le \inf_{t \in H} X_n(t)+\epsilon_n$ with $(\epsilon_n)$ converging to zero in probability. Clearly, if $\epsilon_n=0$, then $\xi_n$ is a minimizer of $X_n$.

These two Argmin-theorems are no longer applicable, when the limit process $X$ has more than one minimizer with positive probability.
This situation is far away to be pathological as it frequently occurs in applications. It typically arises in the analysis of regression functions with a
break-point. Here, the limit $X$ turns out to be two-sided compound Poisson process on $\mathbb{R}$ with drift upwards, which in general has a finitely union of compact intervals as its set of all minimizing points. See for instance, Pons (2003) \cite{pons2003estimation}, Kosorock and Sen (2007) \cite{kosorok2007inference}, Kosorock (2008) \cite{kosorok2008introduction}, Lan et al. (2009) \cite{lan2009change}, Ferger and Klotsche  (2009) \cite{ferger2009estimation} or Albrecht (2020) \cite{albrecht2020}. Notice, that a compound Poisson process with upwards drift has with probability one a smallest and largest minimizing point, say $\sigma$ and $\tau$, which do not coincide: $\sigma < \tau$. In Ferger (2004) \cite{ferger2004continuous} we give a first solution for $H=\mathbb{R}$ in the non-unique case. Under $(\ref{Hoffmann})$ and $(\ref{stochasticallybounded})$ it is shown that for all $x \in \mathbb{R}$,
\begin{equation} \label{sigma}
 \limsup_{n \rightarrow \infty} \mathbb{P}^*(\xi_n \le x) \le \mathbb{P}(\sigma \le x)
\end{equation}
and
\begin{equation} \label{tau}
 \liminf_{n \rightarrow \infty} \mathbb{P}_*(\xi_n < x) \ge \mathbb{P}(\tau < x),
\end{equation}
where $\sigma$ and $\tau$ are the smallest and largest minimizer of $X$. These results carry over without problems to
$\textbf{F}=C(\mathbb{R})$, the space of all continuous real functions on $\mathbb{R}$ equipped with the topology of uniform convergence on compacta.
Here, the convergence in (\ref{Hoffmann}) reduces to usual distributional convergence, i.e.
\begin{equation} \label{CR}
 X_n \stackrel{\mathcal{D}}{\rightarrow} X \quad \text{ in } C(\mathbb{R}).
\end{equation}
Moreover, the outer and inner probability occurring in (\ref{stochasticallybounded}), (\ref{sigma}) and (\ref{tau}) can be replaced by $\mathbb{P}$, because all participating maps are measurable. Notice that similarly as for $\textbf{F}=l^\infty(H)$ the requirement (\ref{CR}) is equivalent to
the convergence of the restrictions on compact intervals:
\begin{equation} \label{CK}
 X_n \stackrel{\mathcal{D}}{\rightarrow} X \quad \text{ in } C[-a,a] \quad \text{for all } a>0,
\end{equation}
confer Theorem 5 in Whitt (1970) \cite{whitt1970weak}. This is useful to know, because there are several sufficient conditions for (\ref{CK}), confer
Theorems 8.1 and 8.2 and especially the nice moment-criterion in Theorem 12.3 in Billingsley (1968) \cite{Billingsley1968}.

The space $\textbf{F}=\mathbb{R}^\mathbb{Z}$ of all sequences $f=(f(k): k \in \mathbb{Z})$ with the product topology occurs naturally in change-point analysis. Here, the counterpart of (\ref{Hoffmann}) is the convergence of the finite dimensional distributions of $X_n$ to $X$.
Ferger \cite{ferger2004continuous} shows this and (\ref{stochasticallybounded}) for a class of change-point estimators.
As a result we obtain limit theorems of type (\ref{sigma}) and (\ref{tau}) without asterisks (by measurability as above), where the limit process $X$ is a two-sided random walk on the integers.

Now, in the above mentioned examples from regression analysis the trajectories of the involved processes live in the space $D(\mathbb{R})$ of all
right-continuous functions $f:\mathbb{R} \rightarrow \mathbb{R}$ with left limits endowed with the Skorokhod-topology.
Also in this case Ferger \cite{ferger2004continuous}
proves the following Argmin-theorem: If the restrictions of the $X_n$ to compact intervals $[-a,a]$ converge, i.e.
\begin{equation} \label{Dcompact}
 X_n \stackrel{\mathcal{D}}{\rightarrow} X \quad \text{ in } D[-a,a] \quad \text{for all real } a>0,
\end{equation}
and $\xi_n=O_\mathbb{P}(1)$ (stochastical boundedness), then
\begin{equation} \label{sigmatau}
\limsup_{n \rightarrow \infty} \mathbb{P}(\xi_n \le x) \le \mathbb{P}(\sigma \le x) \text{ and }   \liminf_{n \rightarrow \infty} \mathbb{P}(\xi_n < x) \ge \mathbb{P}(\tau < x).
\end{equation}
Analogously to continuous stochastic processes (\ref{Dcompact}) is equivalent to $ X_n \stackrel{\mathcal{D}}{\rightarrow} X \text{ in } D(\mathbb{R})$, confer Theorem 16.7 in Billingsley \cite{Billingsley1999}. As a consequence, one can use e.g. the easy to handle criterion in Theorem 13.5 in Billingsley (1999) \cite{Billingsley1999} to ensure (\ref{Dcompact}).

If $\mathcal{O}_<:=\{(-\infty,x): x \in \mathbb{R}\} \cup \{\emptyset,\mathbb{R}\}$ and $\mathcal{O}_>:=\{(x,\infty): x \in \mathbb{R}\} \cup \{\emptyset,\mathbb{R}\}$ denote the left-order topology and right-order topology on $\mathbb{R}$, then we obtain from (\ref{sigma}) and (\ref{tau}) that
$\xi_n \rightsquigarrow \sigma$ in $(\mathbb{R},\mathcal{O}_>)$ and $\xi_n \rightsquigarrow \tau$ in $(\mathbb{R},\mathcal{O}_<)$. Recall that our notation expresses that $\mathcal{O}_>$ and $\mathcal{O}_<$  are the respective underlying topologies. For $\textbf{F} \in \{D(\mathbb{R}),C(\mathbb{R}), \mathbb{R}^\mathbb{Z}\}$ one gets $\xi_n \stackrel{\mathcal{D}}{\rightarrow} \sigma$ in $(\mathbb{R},\mathcal{O}_>)$ and $\xi_n \stackrel{\mathcal{D}}{\rightarrow} \tau$ in $(\mathbb{R},\mathcal{O}_<)$, i.e. traditional distributional convergence, however not with respect to the natural topology $\mathcal{O}_n$ on $\mathbb{R}$, but with respect to the weaker topologies $\mathcal{O}_>$ and $\mathcal{O}_<$.
Assume that $\mathbb{E}[\sigma]= \mathbb{E}[\tau]$ as for instance when $\sigma \stackrel{\mathcal{D}}{=} \tau$. Then actually $\sigma=\tau$ almost surely (a.s.), since $0=\mathbb{E}[\tau]-\mathbb{E}[\sigma]=\mathbb{E}[\tau-\sigma]$ and $\tau-\sigma \ge 0$ by definition. So, $X$ has a unique minimizing point a.s. if and only if $\sigma \stackrel{\mathcal{D}}{=} \tau$. In particularly, in this situation
$\xi_n \rightsquigarrow \sigma$ in $(\mathbb{R},\mathcal{O}_>)$ and $\xi_n \rightsquigarrow \sigma$ in $(\mathbb{R},\mathcal{O}_<)$, whence
by  Example 3.5 in Ferger (2024) \cite{ferger2024semi}, $\xi_n \stackrel{\mathcal{D}}{\rightarrow} \sigma$ in $(\mathbb{R},\mathcal{O}_n)$ so that we obtain classical convergence in distribution.

In $\mathbb{R}^d$ with dimension $d \ge 2$ the notion of smallest and largest minimizer is not given a priori.
Theorem 3.2 of Seijo and Sen (2011) \cite{seijo2011continuous} goes exactly in this direction for $\textbf{F}=D(\mathbb{R}^d)$, the multivariate Skorokhod-space. In Definition 2.4 \cite{seijo2011continuous} they introduce the functionals sargmin and largmin, which by using the idea of lexicographic order give a smallest and largest minimizing point of a function $f \in D(\mathbb{R}^d)$. For $X_n$ and $X$ with some specific trajectories in $D(\mathbb{R}^d)$ and so-called associated jump processes $\Gamma_n$ and $\Gamma$ it is shown that if $(X_n,\Gamma_n)$ converges in distribution to $(X,\Gamma)$ and $(\mbox{sargmin}(X_n),\mbox{largmin}(X_n))=O_\mathbb{P}(1),$ then in fact
\begin{equation} \label{SeijoSen}
 (\mbox{sargmin}(X_n),\mbox{largmin}(X_n)) \rightsquigarrow (\mbox{sargmin}(X),\mbox{largmin}(X)) \quad \text{in } \mathbb{R}^d \times \mathbb{R}^d.
\end{equation}
An alternative approach to Seijo and Sen  \cite{seijo2011continuous} that does not use the concept of smallest and largest minimizer in $\mathbb{R}^d$ is introduced in Ferger (2015) \cite{ferger2015arginf}.
We prove that, if $X_n \stackrel{\mathcal{D}}{\rightarrow} X$ in $D(\mathbb{R}^d)$ endowed with the multivariate Skorokhod-topology, then
\begin{equation} \label{compact}
 \limsup_{n \rightarrow \infty} \mathbb{P}(\xi_n \in K) \le \mu(K) \quad \text{ for all compact } K \subseteq \mathbb{R}^d,
\end{equation}
where $\mu(K)=\mathbb{P}(A(X) \cap K \neq \emptyset)$ and $A(X)$ is equal to the set of all minimizing points of the process $X$.
The set-function $\mu$ is called \emph{capacity functional} of the random closed set $A(X)$. In particularly, it is a \emph{Choquet-capacity}.
Now, to every Choquet-capacity $\mu$ there exists a random closed set $C$ such that $\mu(K)=\mathbb{P}(C \cap K \neq \emptyset)$ for all compact $K$. Further, every Choquet-capacity can be extended to the Borel-$\sigma$ algebra $\mathcal{B}(\mathbb{R}^d)$ on $\mathbb{R}^d$ such that $\mu(B)=\mathbb{P}(C \cap B \neq \emptyset)$ for all Borel-sets $B$, where $\{C \cap B \neq \emptyset \} \in \mathcal{A}$. The extension in general is not a probability measure, because it lacks $\sigma$-additivity. For facts on random closed sets and Choquet-capacities we refer to Molchanov (2017) \cite{Molchanov2017}. By Proposition 3.3 in Ferger \cite{ferger2015arginf} $\mu$ is a probability measure if and only if the corresponding random closed set $C$ is equal to a singleton $\{\xi\}$ almost surely for some random variable $\xi$.

Notice that for (\ref{compact}) beside $X_n \stackrel{\mathcal{D}}{\rightarrow} X$  no further requirement is needed.
However, if in addition $(\xi_n)$ is stochastically bounded, then (\ref{compact}) holds even for all closed sets:
\begin{equation} \label{closed}
 \limsup_{n \rightarrow \infty} \mathbb{P}(\xi_n \in F) \le \mu(F) \quad \text{ for all closed } F \subseteq \mathbb{R}^d.
\end{equation}
Since (\ref{closed}) formally looks exactly like the characterisation of weak convergence in the Portmanteau-Theorem and moreover $\mu$ uniquely corresponds to $A(X)$, we say that the sequence $(\xi_n)$ of \textbf{points} converges in distribution to the \textbf{set} $A(X)$ of all minimizing points. Finally,
if $\mu$ is a probability measure, then as we have seen above $A(X)=\{\xi\}$ a.s., whence
$$
\mu(B)=\mathbb{P}(A(X)\cap B \neq \emptyset)=\mathbb{P}(\{\xi\} \cap B \neq \emptyset)=\mathbb{P}(\xi \in B) \quad \forall\; B \in \mathcal{B}(\mathbb{R}^d).
$$
Consequently, (\ref{closed}) shows in this case that
\begin{equation} \label{dconv}
\xi_n \stackrel{\mathcal{D}}{\rightarrow} \xi \quad \text{in } \mathbb{R}^d.
\end{equation}
Notice that $A(X)=\{\xi\}$ a.s. means exactly that $X$ has a unique minimizing point (namely $\xi$) with probability one.

Actually, in Ferger \cite{ferger2015arginf} we consider more generally non-empty closed sets $\varphi_n$, which are a.s. subsets of all minimizing points $A(X_n)$ of $X_n$. Then given $X_n \stackrel{\mathcal{D}}{\rightarrow} X$ in $D(\mathbb{R}^d)$ it is shown that
\begin{equation} \label{dconvupperFell}
 \varphi_n \stackrel{\mathcal{D}}{\rightarrow} A(X) \quad \text{in } (\mathcal{F},\tau_{\text{uF}}),
\end{equation}
where $\mathcal{F}$ is the family of all closed subsets of $\mathbb{R}^d$ and $\tau_{\text{uF}}$ is the upper Fell-topology on $\mathcal{F}.$ It follows from the definition of $\tau_{\text{uF}}$ (see the next section) in combination with the Portmanteau-Theorem that (\ref{dconvupperFell}) is equivalent to
\begin{equation} \label{compactmissingsets}
 \limsup_{n \rightarrow \infty}\mathbb{P}\Big(\bigcap_{K \in \mathcal{K}^*} \{\varphi_n \cap K \neq \emptyset\}\Big) \le \mathbb{P}\Big(\bigcap_{K \in \mathcal{K}^*} \{A(X) \cap K \neq \emptyset\}\Big) \quad \text{for all } \mathcal{K}^* \subseteq \mathcal{K},
\end{equation}
where $\mathcal{K}$ denotes the class of all compact subsets in $\mathbb{R}^d$.
If in addition the sequence of subsets $(\varphi_n)$ is \emph{stochastically bounded} in the sense that
\begin{equation} \label{asymptsubset}
 \lim_{k \rightarrow \infty} \limsup_{n \rightarrow \infty} \mathbb{P}(\varphi_n \nsubseteq [-k,k]^d)=0,
\end{equation}
then
\begin{equation} \label{dconvupperVietoris}
 \varphi_n \stackrel{\mathcal{D}}{\rightarrow} A(X) \quad \text{in } (\mathcal{F},\tau_{\text{uV}}),
\end{equation}
where $\tau_{\text{uV}}$ is the upper Vietoris-topology on $\mathcal{F}.$ This is the same as
\begin{equation} \label{closedmissingsets}
 \limsup_{n \rightarrow \infty}\mathbb{P}\Big(\bigcap_{F \in \mathcal{F}^*} \{\varphi_n \cap F \neq \emptyset\}\Big) \le \mathbb{P}\Big(\bigcap_{F \in \mathcal{F}^*} \{A(X) \cap F \neq \emptyset\}\Big) \quad \text{for all } \mathcal{F}^* \subseteq \mathcal{F}.
\end{equation}
Finally, if furthermore $A(X)=\{\xi\}$ a.s., then
\begin{equation} \label{dconvFell}
 \varphi_n \stackrel{\mathcal{D}}{\rightarrow} A(X) \quad \text{in } (\mathcal{F},\tau_F),
\end{equation}
where $\tau_F$ denotes the Fell-topology. When comparing the convergence results (\ref{dconvupperFell}), (\ref{dconvupperVietoris}) and (\ref{dconvFell}) notice that $\tau_{\text{uF}} \subseteq \tau_{\text{uV}}$ and $\tau_{uF} \subseteq \tau_F$, so that the statements become stronger each time.

Consider the special case that $\varphi_n=\{\xi_n\}$. Then the choice $\mathcal{K}^*=\{K\}$ in (\ref{compactmissingsets}) immediately gives (\ref{compact}). Similarly, $\mathcal{F}^*=\{F\}$ in (\ref{closedmissingsets}) directly yields (\ref{closed}).

In stochastic optimisation it is convenient to work with the space $\textbf{F}= S(H)$ of  all lower semicontinuous functions $f:H \rightarrow \overline{\mathbb{R}}$ equipped with the epi-topology $\mathcal{T}_e$, confer e.g. Salinetti and Wets (1986) \cite{salinetti1986convergence}, Pflug (1992, 1995) \cite{Pflug1992, pflug1995asymptotic}, Vogel (2005, 2006) \cite{vogel2005qualitative,vogel2006semiconvergence} or Gersch (2006) \cite{gersch2006convergence}.  For $\epsilon \ge 0$ let $A(f,\epsilon)$ be the set of all $\epsilon$-optimal solutions. Assume that
\begin{equation} \label{dconvinSH}
 X_n \stackrel{\mathcal{D}}{\rightarrow} X \quad \text{in } (S(H),\mathcal{T}_e),
\end{equation}
where $H$ is a locally compact second countable Hausdorff-space. If moreover
\begin{equation} \label{epsilonconverges}
\epsilon_n \stackrel{\mathcal{D}}{\rightarrow} \epsilon  \text{ with } \epsilon \text{ is constant a.s.},
\end{equation}
then by Theorem 5 in Ferger (2025) {\cite{ferger2025epi}
\begin{equation} \label{optimalsolutionsinupperFell}
 A(X_n,\epsilon_n) \stackrel{\mathcal{D}}{\rightarrow} A(X,\epsilon) \quad \text{in } (\mathcal{F}, \tau_{\text{uF}})
\end{equation}
with $\mathcal{F}$ is equal to the family of all closed subsets in $H$. Next, suppose for every $\eta >0$ there exists a compact $K \subseteq H$ such that
$$
 \limsup_{n \rightarrow \infty}\mathbb{P}(A(X_n,\epsilon_n) \nsubseteq K) \le \eta.
$$
Then it follows from (\ref{optimalsolutionsinupperFell}) and Corollary 2.2 in Ferger (2024) \cite{ferger2024weak} that
\begin{equation} \label{optimalsolutionsinupperVietoris}
 A(X_n,\epsilon_n) \stackrel{\mathcal{D}}{\rightarrow} A(X,\epsilon) \quad \text{in } (\mathcal{F}, \tau_{\text{uV}}).
\end{equation}
By Proposition 2.1 and Remark 2.3 in Ferger \cite{ferger2024weak} we know that
(\ref{optimalsolutionsinupperFell}) and (\ref{optimalsolutionsinupperVietoris}) also hold for non-empty subsets $\varphi_n \subseteq A(X_n,\epsilon_n)$ and that the equivalent characterisations (\ref{compactmissingsets}) and (\ref{closedmissingsets}) are valid analogously.

If (\ref{epsilonconverges}) holds with $\epsilon =0$ and  $X$ has at most one minimizing point $\xi$ a.s., then by Theorem 6 and Remark 5 in Ferger \cite{ferger2025epi}
\begin{equation} \label{optimalsolutionsinFell}
 A(X_n,\epsilon_n) \stackrel{\mathcal{D}}{\rightarrow} \{\xi\} \quad \text{in } (\mathcal{F}, \tau_{F}).
\end{equation}
Clearly, for the constant sequence $(\epsilon_n)\equiv 0$, i.e. $A(X_n,\epsilon_n)=A(X_n)$ for all $n \in \mathbb{N}$, the requirement (\ref{epsilonconverges}) is automatically fulfilled.

As to $\epsilon_n$-optimal solutions $\xi_n$ the special choice $\varphi_n:=\{\xi_n\}$ yields that (\ref{dconvinSH}) and (\ref{epsilonconverges}) entail
\begin{equation} \label{compactH}
 \limsup_{n \rightarrow \infty} \mathbb{P}(\xi_n \in K) \le \mu_\epsilon(K) \quad \text{ for all compact } K \subseteq H,
\end{equation}
where $\mu_\epsilon$ is the capacity-functional of $A(X,\epsilon)$, and
\begin{equation} \label{closedH}
 \limsup_{n \rightarrow \infty} \mathbb{P}(\xi_n \in F) \le \mu_\epsilon(F) \quad \text{ for all closed } F \subseteq H,
\end{equation}
provided $(\xi_n)$ is uniformly tight, i.e. (\ref{uniformtightness}) holds without asterisks. If in addition $\epsilon=0$ and
$A(X) \subseteq \{\xi\}$ a.s., then
\begin{equation} \label{dconvinH}
\xi_n \stackrel{\mathcal{D}}{\rightarrow} \xi \text{ in } H.
\end{equation}
For the last two statements confer Theorem 7 in Ferger \cite{ferger2025epi}.

Back to $H=\mathbb{R}^d$. Here, the situation becomes very easy when the stochastic processes $X_n$ and $X$ are not only lower semicontinuous, but also convex. Indeed, Ferger (2021) \cite{ferger2021continuous} shows that then the basic requirement (\ref{dconvinSH}) can be replaced by the much weaker assumption
\begin{equation} \label{fidis}
 (X_n(t_1),\ldots,X_n(t_k)) \stackrel{\mathcal{D}}{\rightarrow}  (X(t_1),\ldots,X(t_k)) \quad \text{in } \overline{\mathbb{R}}^k
\end{equation}
for all $t_1,\ldots,t_k \in D$, where $D$ is a dense subset of $\mathbb{R}^d$ (\emph{convergence of the finite dimensional distributions on $D$}).
But not only this! Also, any condition of stochastic boundedness can be omitted. So, if only  (\ref{fidis}) and (\ref{epsilonconverges}) hold, then (\ref{optimalsolutionsinupperVietoris})-(\ref{dconvinH}) follow. In particular,
we obtain as special cases earlier results by Geyer (1996) \cite{geyer1996asymptotics}, Davis, Knight and Liu (1992) \cite{davis1992m} and Hjort and Pollard (2011) \cite{hjort3806asymptotic}.

Moreover, by Theorem 1.4 in Ferger \cite{ferger2021continuous} there exists a stronger version of (\ref{optimalsolutionsinFell}):
\begin{equation} \label{optimalsolutionsinVietoris}
 A(X_n,\epsilon_n) \rightsquigarrow \{\xi\} \quad \text{in } (\mathcal{F}, \tau_{V}).
\end{equation}
Here, $\tau_V$ is the Vietoris topology on $\mathcal{F}$, which is the strongest one among all other hyperspace topologies occurring so far:
$\tau_V \supseteq \tau_{\text{uV}}$ and $\tau_V \supseteq \tau_F \supseteq \tau_{\text{uF}}$.\\

In this paper, we consider more generally a multidimensional vector $(\xi_{\alpha,1},\allowbreak \ldots,\allowbreak \xi_{\alpha,k},\allowbreak \sigma_\alpha)$ consisting of
a.s.\ minimizing points $\xi_{\alpha,j}$ of $X_\alpha^{(j)}, 1 \leq j \leq k$. Here, each $X_\alpha^{(j)}$ is a multivariate c\`{a}dl\`{a}g stochastic process
with trajectories in $D(\mathbb{R}^{d_j}), \; d_j \in \mathbb{N}$. Moreover, $\sigma_\alpha$ is a random variable in some metric space $S$ and the index $\alpha$ runs through a directed set $(I,\le)$. Thus, $(\xi_{\alpha,1},\ldots,\xi_{\alpha,k},\sigma_\alpha)_{\alpha \in I}$ is a net in the product space $H=\mathbb{R}^{d_1} \times \cdots \times \mathbb{R}^{d_k} \times S$. Suppose that each $(X_\alpha^{(j)})_{\alpha \in I}$ is tight and that
$(X_\alpha^{(1)},\ldots,X_\alpha^{(k)},\sigma_\alpha) \rightarrow_{\text{fd}} (X^{(1)},\ldots,X^{(k)},\sigma)$, where $\rightarrow_{\text{fd}}$ means convergence of the finite-dimensional distributions. If in addition $(\xi_{\alpha,1},\ldots,\xi_{\alpha,k})_{\alpha \in I}$ is stochastically bounded, then we show in Theorem \ref{ExtendedArginfTheorem} that:
\begin{equation} \label{limsupmu}
\limsup_{\alpha} \mathbb{P} \left( \xi_{\alpha,1} \in F_1, \hdots, \xi_{\alpha,k} \in F_k, \sigma_{\alpha} \in B \right) \le \mu(F_1 \times \cdots \times F_k \times B)
\end{equation}
for all closed $F_j \subseteq \mathbb{R}^{d_j}, 1 \le j \le k,$ and all $\sigma$-continuity-sets $B \subseteq S$. Here, $\mu$ is the capacity-functional
of the random closed set $$C:= A(X^{(1)}) \times \cdots \times A(X^{(k)}) \times \{\sigma\}.$$ That $C$ in fact is a random closed set in
$H= \mathbb{R}^{d_1} \times \cdots \times \mathbb{R}^{d_k} \times S$ endowed with the product-topology follows from Theorem 1.3.25 in Molchanov \cite{Molchanov2017}. Recall that $\mu(E)=\mathbb{P}(C \cap E \neq \emptyset)$ for all Borel-sets $E$ in $H$. Moreover, it is also proved in Theorem \ref{ExtendedArginfTheorem} that
\begin{equation} \label{liminfnu}
\liminf_{\alpha} \mathbb{P} \left( \xi_{\alpha,1} \in G_1, \hdots, \xi_{\alpha,k} \in G_k, \sigma_{\alpha} \in B \right) \ge \nu(G_1 \times \cdots \times G_k \times B)
\end{equation}
for all open $G_j \subseteq \mathbb{R}^{d_j}, 1 \le j \le k,$ and all $\sigma$-continuity-sets $B \subseteq S$. Here, $\nu$ is the \emph{containment-functional} of $C$, i.e. $\nu(E):=\mathbb{P}(C \subseteq E)$. It is related to $\mu$ by $\nu(E)=1-\mu(E^C)$ with $E^C:= H \setminus E$ the complement of $E$ in $H$. Moreover, one sees immediately that
$\nu \le \mu$ (provided $C$ is a.s. non-empty as it is in our case.) Notice that (\ref{liminfnu}) does not follow from (\ref{limsupmu}) by complementation.
Comparing  (\ref{limsupmu}) and (\ref{liminfnu}) with the limsup- and liminf-characterization in the Portmanteau -Theorem makes us to say that
the points $(\xi_{\alpha,1},\ldots,\xi_{\alpha,k},\sigma_\alpha)$, $\alpha \in I$, converge in distribution to the set $C$, formally written as:
\begin{equation} \label{convergencetoaset}
 (\xi_{\alpha,1},\ldots,\xi_{\alpha,k},\sigma_\alpha) \stackrel{\mathcal{D}}{\rightarrow} A(X^{(1)}) \times \cdots \times A(X^{(k)}) \times \{\sigma\}.
\end{equation}
If $\xi_{\text{min}}^{(j)}:=\text{sargmax}(X^{(j)})$ and $\xi_{\text{max}}^{(j)}:=\text{largmax}(X^{(j)})$ are the smallest and largest minimizer of $X^{(j)}, 1 \le j \le k$, then in Corollary \ref{CorollaryUnivariateExtendedArginf} we obtain as a special case that
\begin{equation} \label{limsupximin}
\limsup_{\alpha} \mathbb{P} \left( \xi_{\alpha,j} \leq x_j, 1 \le j \le k, \sigma_{\alpha} \in B \right) \leq \mathbb{P} \left( \xi_{\text{\text{min}}}^{(j)} \leq x_j, 1 \le j \le k, \sigma \in B \right),
\end{equation}
and
\begin{equation} \label{liminfximax}
\liminf_{\alpha} \mathbb{P} \left( \xi_{\alpha,j} < x_j, 1 \le j \le k, \sigma_{\alpha} \in B \right) \geq \mathbb{P} \left( \xi_{\text{max}}^{(j)} < x_j, 1 \le j \le k, \sigma \in B \right),
\end{equation}
for all $x_j \in \mathbb{R}^{d_j}$, $1 \leq j \leq k$ and for all $\sigma$-continuity sets $B$. Here, the relations $\le$ and $<$ in the multi-dimensional euclidean space are defined componentwise.


Finally, if each process $X^{(j)}$ possesses a.s. an unique minimizing point $\xi_j$, then the limit-set in (\ref{convergencetoaset}) shrinks to the singleton $\{\xi_1\} \times \cdots \times \{\xi_k\} \times \{\sigma\}$ and as a consequence usual convergence in distribution follows, see Corollary \ref{CorollaryExtendedArginf}:
$$
 (\xi_{\alpha,1},\ldots,\xi_{\alpha,k},\sigma_\alpha) \stackrel{\mathcal{D}}{\rightarrow} (\xi_1,\cdots, \xi_k, \sigma) \text{ in } \mathbb{R}^{d_1}\times \cdots \times \mathbb{R}^{d_k} \times S.
$$

The adjunction of $\sigma_\alpha$ in all of our findings might seem artificial. However, in section 3 we show how useful this is
in regression analysis. On the other hand, if in fact $\sigma_\alpha$ is omitted, then we obtain the corresponding results for
$(\xi_{\alpha,1},\ldots,\xi_{\alpha,k})$ under a weaker assumption, confer Remark \ref{withoutsigma}.

\section{Extended Argmin-Theorems} \label{SectionExtendedArginf}
For a natural number $d$ let $X=\{X(t): t \in \mathbb{R}^d\}$ be a real-valued stochastic process defined on some probability space $(\Omega,\mathcal{A},\mathbb{P})$ with trajectories in the \emph{multivariate Skorokhod-space} $D=D(\mathbb{R}^d)$. For the definition of $D$ we consider $d$-tuples $R=(R_1,\ldots,R_d) \in \{<,\ge\}^d$ with the usual relations $<$ and $\ge$ in $\mathbb{R}$. If $t=(t_1,\ldots,t_d) \in \mathbb{R}^d$ is a point in the euclidean space, then
\begin{align*}
 Q_R := Q_R(t) := \{s \in \mathbb{R}^d: s_i R_i t_i, 1 \le i \le d \}
\end{align*}
is the \emph{R-quadrant of t}. Given a function $f:\mathbb{R}^d \rightarrow \mathbb{R}$ the quantity
\begin{align*}
 f(t+R) := \lim_{s\rightarrow t, s \in Q_R(t)} f(s)
\end{align*}
is called the \emph{R-quadrant-limit of f at t}. Then $D$ consists of all functions $f$ such that for each $t \in \mathbb{R}^d$
\begin{enumerate} [(a)]
\renewcommand{\theenumi}{\alph{enumi}}
\renewcommand{\labelenumi}{\theenumi)}
\item \label{Lag} $f(t+R)$  exists for all $R \in \{<, \ge\}^d$\;,
\item \label{Cad} $f(t+R) = f(t)$ for  $R=(\ge,\ldots,\ge)$.
\end{enumerate}
Relations \eqref{Lag} and \eqref{Cad} extend the notions ,,limits from below'' and ,,continuous from above'' from the univariate case ($d=1$) to the multivariate one. Therefore it is convenient to call $f \in D$ a \emph{c\`{a}dl\`{a}g function} (continue \`{a} droite limite \`{a} gauche). $D$ endowed with the \emph{Skorokhod-metric }$s$ is a complete separable metric space, confer Lagodowski and Rychlik (1986) \cite{Lagodowski1986}, p. 332. The pertaining Borel-$\sigma$-algebra
$\mathcal{D}$ is generated by the sets of all cylinders, confer Theorem 2 of Lagodowski and Rychlik \cite{Lagodowski1986}. Therefore, $X$ can be identified with a random element $X:(\Omega,\mathcal{A})\rightarrow (D,\mathcal{D})$.\\
\, \\
Let
\begin{equation}
 \mbox{Argmin}(f) \equiv A(f) := \{ t \in \mathbb{R}^d: \min_{R \in \{<,\ge\}^d} f(t+R)=\inf_{s \in \mathbb{R}^d} f(s) \}, \quad f \in D, \label{arginf}
\end{equation}
be the set of all \emph{minimizing points} of $f$. It should be noted that we call a point $t \in A(f)$ \textbf{minimizing} point even though the function $f$ in general does not attain its minimal value at that point. One reason for this is that the pertaining lower semicontinuous regularisation $\bar{f}$ of $f$ is minimized at each $t \in A(f)$ and conversely every (proper) minimizer $t$ of $\bar{f}$ lies in $A(f)$, confer Lemma 2.2 in Ferger \cite{ferger2015arginf}. Therefore, $A(f)$ is a closed subset of $\mathbb{R}^d$ (possibly empty). So, if the functional $A:D \rightarrow \mathcal{F}$ is applied to
$X$ one obtains a map from $\Omega$ into the family $\mathcal{F}$ of all closed subsets of $\mathbb{R}^d$ including the empty set $\emptyset$:
$A(X)= A \circ X: \Omega \rightarrow \mathcal{F}.$ We endow $\mathcal{F}$ with a topology suitable for our purposes. To this end introduce for every subset $C \subseteq \mathbb{R}^d$ the system
$\mathcal{M}(C):=\{F \in \mathcal{F}: F \cap C = \emptyset\}$ of all \emph{missing sets} of $C$
and for later use
$\mathcal{H}(C):=\{F \in \mathcal{F}: F \cap C \neq \emptyset\}$ of all \emph{hitting sets} of $C$.
Put
$$
\mathcal{S}:= \{\mathcal{M}(K): K \in \mathcal{K}\}  \subseteq 2^\mathcal{F}.
$$
Then the topology  on $\mathcal{F}$ generated by $\mathcal{S}$ is called \emph{upper Fell-topology} and denoted by $\tau_{\text{uF}}$.
It induces the pertaining Borel-$\sigma$ algebra $\mathcal{B}_{\text{uF}} := \sigma(\tau_{\text{uF}})$. Proposition 2.7 in Ferger \cite{ferger2015arginf} shows that
$A(X)$ is $\mathcal{A}-\mathcal{B}_{\text{uF}}$ measurable. Such maps are called \emph{random closed set} in $\mathbb{R}^d$.
If $A(X)$ is non-empty a.s., then the \emph{Fundamental selection theorem}, confer Molchanov \cite{Molchanov2017}, guarantees the existence of a Borel-measurable map
$\xi:(\Omega,\mathcal{A}) \rightarrow \mathbb{R}^d$ such that
$\xi \in A(X)$ a.s. This random variable $\xi$ is called a \emph{measurable selection} of $A(X)$. In particular, $\xi$ is a minimizing point of the stochastic process $X$. \\
\, \\
For every $t \in \mathbb{R}^d$ the projection (evaluation map) $\pi_t: D \rightarrow \mathbb{R}$ is defined by $\pi_t(f) := f(t)$. If $T = \{ t_1, \hdots, t_k \} \subseteq \mathbb{R}^d$, then $\pi_T := (\pi_{t_1}, \hdots, \pi_{t_k})$. Finally, we introduce
\begin{align*}
T_X :=  \left\{ t \in \mathbb{R}^d :  \pi_t \text{ is continuous at } X \text{ a.s.} \right\}.
\end{align*}
To emphasize the dependence on the dimension $d$ we write $D_d, \mathcal{D}_d$ and $\mathcal{F}_d$ for $D, \mathcal{D}$ and $\mathcal{F}$. But for notational convenience, we avoid the index for $A$ and $\tau_{\text{uF}}$, the readers should keep in mind that they depend on $d$ as well.\\

\noindent
Our first result is the starting point for the Extended-Argmin Theorem. It is a generalisation of Theorem 5.1 in Ferger (2010) \cite{Ferger2010}.

\begin{proposition} \label{PropositionCriteriaForJointWeakConvInProdSpaceForD(Rd)}
For finitely many natural numbers $d_1,\hdots,d_k$ let
\begin{itemize}
\item $(X_{\alpha}^{(j)})$ be a net of random variables in $D_{d_j}$ for every $1 \leq j \leq k$,
\item $X^{(j)}$ a random variable in $D_{d_j}$ for every $1 \le j \le k$,
\item $(\sigma_{\alpha})$ a net of random variables in a separable and complete metric space $S$ with Borel-$\sigma$-algebra $\mathcal{B}(S)$,
\item $\sigma$ a random variable in $S$.
\end{itemize}
Suppose the nets meet the conditions
\begin{enumerate} [(i)]
\item \label{Cond1PropositionCriteriaForJointWeakConvInProdSpaceForD(Rd)} $(X_\alpha^{(j)})_{\alpha \in I}$ is tight for each fixed $1 \le j \le k.$
\item \label{Cond2PropositionCriteriaForJointWeakConvInProdSpaceForD(Rd)}
\begin{align*}
&\left( \pi_{T_1} \left( X_{\alpha}^{(1)} \right),\hdots, \pi_{T_k} \left( X_{\alpha}^{(k)} \right), \sigma_{\alpha} \right)  \xrightarrow{\mathcal{L}} \left( \pi_{T_1} \left( X^{(1)} \right), \hdots, \pi_{T_k} \left( X^{(k)} \right), \sigma \right)
\end{align*}
in  $\mathbb{R}^{\vert T_1 \vert} \times \hdots \times \mathbb{R}^{\vert T_k \vert} \times S$ for all finite $T_j \subseteq T_{X^{(j)}}$, $1 \leq j \leq k$.
\end{enumerate}
Then
\begin{align*}
\left( X_{\alpha}^{(1)}, \hdots, X_{\alpha}^{(k)}, \sigma_{\alpha} \right) \xrightarrow{\mathcal{L}} \left( X^{(1)}, \hdots, X^{(k)}, \sigma \right) \text{ in } D_{d_1} \times \hdots \times D_{d_k} \times S.
\end{align*}
\end{proposition}

\begin{proof}
For notational convenience let $k = 2$. We write short $X_{\alpha}$, $X$, $Y_{\alpha}$, $Y$ for $X_{\alpha}^{(1)}$, $X^{(1)}$, $X_{\alpha}^{(2)}$, $X^{(2)}$, respectively. It follows from \eqref{Cond2PropositionCriteriaForJointWeakConvInProdSpaceForD(Rd)} and the Continuous Mapping Theorem (CMT) that $\sigma_{\alpha} \xrightarrow{\mathcal{L}} \sigma$ in $S$, whence $(\sigma_{\alpha})$ in particularly is relatively compact and thus $(\sigma_{\alpha})$ is tight by Prokhorov's theorem. Using \eqref{Cond1PropositionCriteriaForJointWeakConvInProdSpaceForD(Rd)} and Tikhorov's theorem we see that $(X_{\alpha}, Y_{\alpha}, \sigma_{\alpha})$ is tight and by another application of Prokhorov's theorem it is relatively compact. So, if $(\alpha')$ is an arbitrary subnet of $(\alpha)$ there exists a subnet $(\alpha'')$ of $(\alpha')$ such that
\begin{align} \label{TempEq1PropositionCriteriaForJointWeakConvInProdSpaceForD(Rd)}
\left( X_{\alpha''}, Y_{\alpha''}, \sigma_{\alpha''} \right) \xrightarrow{\mathcal{L}} \left( X', Y', \sigma' \right) \text{ in } D_{d_1} \times D_{d_2} \times S,
\end{align}
where the limit in \eqref{TempEq1PropositionCriteriaForJointWeakConvInProdSpaceForD(Rd)} depends on $(\alpha')$. It follows from the CMT that
\begin{equation} \label{TempEq2PropositionCriteriaForJointWeakConvInProdSpaceForD(Rd)}
\begin{aligned}
&\left( \pi_T \left( X_{\alpha''} \right), \pi_U \left( Y_{\alpha''} \right), \sigma_{\alpha''} \right) \xrightarrow{\mathcal{L}} \left( \pi_T (X'), \pi_U (Y'), \sigma' \right) \text{ in } \mathbb{R}^{\vert T \vert} \times \mathbb{R}^{\vert U \vert} \times S
\end{aligned}
\end{equation}
for each finite $T \subseteq T_{X'}$ and $U \subseteq T_{Y'}$. To see that indeed the CMT can be applied, note that $T \subseteq T_{X'}$ implies that $\pi_T$ is continuous at $X'$ a.s. For the same reason $\pi_U$ is continuous at $Y'$ a.s. This enables the application of the CMT.
From \eqref{Cond2PropositionCriteriaForJointWeakConvInProdSpaceForD(Rd)} and another application of the CMT we obtain that
\begin{equation} \label{TempEq3PropositionCriteriaForJointWeakConvInProdSpaceForD(Rd)}
\begin{aligned}
&\left( \pi_T \left( X_{\alpha''} \right), \pi_U(Y_{\alpha''}), \sigma_{\alpha''} \right) \xrightarrow{\mathcal{L}} \left( \pi_T(X), \pi_U(Y), \sigma \right) \text{ in } \mathbb{R}^{\vert T \vert} \times \mathbb{R}^{\vert U \vert} \times S
\end{aligned}
\end{equation}
for each finite $T \subseteq T_{X}$ and $U \subseteq T_{Y}$. Put $T_0 := T_X \cap T_{X'}$, $U_0 := T_Y \cap T_{Y'}$,
\begin{align*}
\mathcal{F}_{T_0} := \left\lbrace \pi_T^{-1}(B) : B \in \mathcal{B} \left( \mathbb{R}^{\vert T \vert} \right), \, T \subseteq T_0, \, T \text{ finite} \right\rbrace,
\end{align*}
and
\begin{align*}
\mathcal{F}_{U_0} := \left\lbrace \pi_U^{-1}(B) : B \in \mathcal{B} \left( \mathbb{R}^{\vert T \vert} \right), \, U \subseteq U_0, \, U \text{ finite} \right\rbrace.
\end{align*}
Since $T_0 \subseteq \mathbb{R}^{d_1}$ and $U_0 \subseteq \mathbb{R}^{d_2}$ both are dense by Lemma \ref{Lemma1Appendix} (in the appendix), Theorem 2 in Lagodowski and Rychlik \cite{Lagodowski1986} ensures that $\mathcal{D}_{d_1} = \sigma (\mathcal{F}_{T_0})$ and $\mathcal{D}_{d_2} = \sigma (\mathcal{F}_{U_0})$. Consequently,
\begin{align*}
\mathcal{D}_{d_1} \otimes \mathcal{D}_{d_2} \otimes \mathcal{B}(S) = \sigma \left( \mathcal{F}_{T_0} \times \mathcal{F}_{U_0} \times \mathcal{B}(S) \right)
\end{align*}
by a result of measure theory, confer Theorem 22.1 in Bauer (2001) \cite{Bauer2001}. Further, $\mathcal{F}_{T_0}$, $\mathcal{F}_{U_0}$ and $\mathcal{B}(S)$ are $\pi$-systems and so is $\mathcal{F}_{T_0} \times \mathcal{F}_{U_0} \times \mathcal{B}(S)$. Thus, the latter is a separating class for $\mathcal{D}_{d_1} \otimes \mathcal{D}_{d_2} \otimes \mathcal{B}(S)$. Now, it follows from \eqref{TempEq2PropositionCriteriaForJointWeakConvInProdSpaceForD(Rd)} and \eqref{TempEq3PropositionCriteriaForJointWeakConvInProdSpaceForD(Rd)} that the distributions of $(X',Y',\sigma')$ and $(X,Y,\sigma)$ coincide on $\mathcal{F}_{T_0} \times \mathcal{F}_{U_0} \times \mathcal{B}(S)$ and therefore on $\mathcal{D}_{d_1} \otimes \mathcal{D}_{d_2} \otimes \mathcal{B}(S)$, i.e.
\begin{align*}
(X',Y',\sigma') \stackrel{\mathcal{L}}{=} (X,Y,\sigma).
\end{align*}
By the subnet-criterion for convergence in topological spaces the result follows. (Recall that convergence in distribution is equivalent to weak convergence of the involved distributions, which in turn is equivalent to convergence in the weak topology.)
\end{proof}

\begin{remark} \label{alldimensionsareequal}
If the dimensions are all equal ($d_1 = d_2 = \hdots = d_k = d$), then condition \eqref{Cond2PropositionCriteriaForJointWeakConvInProdSpaceForD(Rd)} can significantly be weakened to
\begin{enumerate} [(i)]
\renewcommand{\theenumi}{\roman{enumi}*}
\renewcommand{\labelenumi}{(\theenumi)}
\setcounter{enumi}{1}
\item \label{Cond2*PropositionCriteriaForJointWeakConvInProdSpaceForD(Rd)}
\begin{align*}
\left( \pi_T \left(X_{\alpha}^{(1)} \right), \hdots, \pi_T \left( X_{\alpha}^{(k)} \right), \sigma_{\alpha} \right) \xrightarrow{\mathcal{L}} \left(\pi_T \left(X^{(1)} \right), \hdots, \pi_T \left( X^{(k)} \right), \sigma \right)
\end{align*}
in $\mathbb{R}^{k \vert T \vert} \times S$ for all finite $T \subseteq T_{X^{(1)}} \cap \hdots \cap T_{X^{(k)}}$.
\end{enumerate}
So, in contrast to \eqref{Cond2PropositionCriteriaForJointWeakConvInProdSpaceForD(Rd)} the sets $T_1,\hdots,T_k$ are all the same. In application this reduces the amount of work drastically. The proof requires only a minor modification as follows: Again for simplicity, let $k = 2$. Introduce $\pi_{T,T}: D_d \times D_d \rightarrow \mathbb{R}^{2 \vert T \vert}$ given by $\pi_{T,T} (f,g) = (\pi_T(f),\pi_T(g))$. With the same arguments as in the proof of Lemma 5.6 in Ferger \cite{Ferger2010} one shows that for every dense $T_0 \subseteq \mathbb{R}^d$,
\begin{align*}
\left\lbrace \pi_{T,T}^{-1} (B) : B \in \mathcal{B} \left(\mathbb{R}^{2 \vert T \vert} \right), T \subseteq T_0, T \text{ finite} \right\rbrace
\end{align*}
is a separating class for $\mathcal{D}_d \otimes \mathcal{D}_d$. Now we leave it to the reader to modify the proof of Proposition \ref{PropositionCriteriaForJointWeakConvInProdSpaceForD(Rd)}.
\end{remark}

In our main result we use the notation $||\cdot||_\infty$ for the maximum-norm on the euclidian space.

\begin{theorem} \label{ExtendedArginfTheorem}
Assume that the conditions \eqref{Cond1PropositionCriteriaForJointWeakConvInProdSpaceForD(Rd)} and \eqref{Cond2PropositionCriteriaForJointWeakConvInProdSpaceForD(Rd)} of Proposition \ref{PropositionCriteriaForJointWeakConvInProdSpaceForD(Rd)} are fulfilled. (If all dimensions $d_1,\ldots,d_k$ are equal, then it suffices to require \eqref{Cond2*PropositionCriteriaForJointWeakConvInProdSpaceForD(Rd)}.)
For every $1 \le j \le k$ let $A(X_{\alpha}^{(j)}) \neq \emptyset$ a.s. and $\xi_{\alpha,j}$ be a measurable selection of $A(X_{\alpha}^{(j)})$. Put $\xi_{\alpha} := (\xi_{\alpha,1},\hdots, \xi_{\alpha,k})$. If
\begin{enumerate} [(i)]
\setcounter{enumi}{2}
\item \label{Cond3ExtendedArginfTheorem} $\begin{aligned}[t]
\lim_{a \rightarrow \infty} \limsup_{\alpha} \mathbb{P} \left( \Vert \xi_{\alpha} \Vert_{\infty} > a \right) = 0,
\end{aligned}$
\end{enumerate}
then
\begin{equation} \label{LimsupIneqExArginf}
\begin{aligned}
&\limsup_{\alpha} \mathbb{P} \left( \xi_{\alpha,1} \in F_1, \hdots, \xi_{\alpha,k} \in F_k, \sigma_{\alpha} \in B \right) \\
&\qquad \qquad \qquad \leq \mathbb{P} \left( A \big( X^{(1)} \big) \cap F_1 \neq \emptyset, \hdots, A \big(X^{(k)} \big) \cap F_k \neq \emptyset, \sigma \in B \right)
\end{aligned}
\end{equation}
for all closed sets $F_j \subseteq \mathbb{R}^{d_j}$ and for all Borel-sets $B \in \mathcal{B}(S)$ with $\mathbb{P} (\sigma \in \partial B) = 0$, i.e. $B$ is a $\sigma$-continuity set. Moreover,
\begin{equation} \label{LiminfIneqExArginf}
\begin{aligned}
&\liminf_{\alpha} \mathbb{P} \left( \xi_{\alpha,1} \in G_1, \hdots, \xi_{\alpha,k} \in G_k, \sigma_{\alpha} \in B  \right) \\
&\qquad \qquad \qquad \geq \mathbb{P} \left( A \big( X^{(1)} \big) \subseteq G_1, \hdots, A \big( X^{(k)} \big) \subseteq G_k, \sigma \in B \right),
\end{aligned}
\end{equation}
for all open sets $G_j \subseteq \mathbb{R}^{d_j}$  and for all $\sigma$-continuity sets $B$.
\end{theorem}

\begin{proof}
By Proposition 2 in Ferger \cite{ferger2015arginf} the map $A: (D_d,s) \rightarrow (\mathcal{F}_d, \tau_{\text{uF}})$ is sequentially continuous for each $d \in \mathbb{N}$ and therefore by Theorem 7.1.3 in Singh (2019) \cite{Singh2019} it is continuous, because every metric space is first countable. Thus Proposition \ref{PropositionCriteriaForJointWeakConvInProdSpaceForD(Rd)} and the CMT yield that
\begin{equation} \label{TempEq1ExtendedArginfTheorem}
\begin{aligned}
&\left( A \left( X_{\alpha}^{(1)} \right), \hdots, A \left(X_{\alpha}^{(k)} \right), \sigma_{\alpha} \right) \xrightarrow{\mathcal{L}} \left( A \left(X^{(1)} \right), \hdots, A \left( X^{(k)} \right) , \sigma \right) \\
&\text{ in the topological product-space } \mathcal{F}_{d_1} \times \hdots \times \mathcal{F}_{d_k} \times S.
\end{aligned}
\end{equation}
In a first step we assume that the $F_j \subseteq \mathbb{R}^{d_j}$, $1 \le j \le k$, are actually compact. By construction of the upper Fell-topology, $\mathcal{M} (F_j) \in \tau_{\text{uF}}$, whence $\mathcal{H}(F_j)$ is $\tau_{\text{uF}}$-closed. If $B \subseteq S$ is closed, then $\mathcal{H}(F_1) \times \hdots \times \mathcal{H}(F_k) \times B$ is closed in the topological product $\mathcal{F}_{d_1} \times \hdots \times \mathcal{F}_{d_k} \times S$. Obviously,
\begin{align*}
\left\lbrace \xi_{\alpha,j} \in F_j \right\rbrace \subseteq \left\lbrace A \left( X_{\alpha}^{(j)} \right) \in \mathcal{H}(F_j) \right\rbrace,
\end{align*}
for each $1 \le j \le k$, whence it follows from \eqref{TempEq1ExtendedArginfTheorem} with the Portmanteau theorem:
\begin{align*}
&\limsup_{\alpha} \mathbb{P} \left( \xi_{\alpha,j} \in F_j, 1 \le j \le k, \sigma_{\alpha} \in B \right) \\
&\leq \limsup_{\alpha} \mathbb{P} \left( \left( A \left( X_{\alpha}^{(1)} \right),\hdots, A \left(  X_{\alpha}^{(k)} \right), \sigma_{\alpha} \right) \in \mathcal{H}(F_1) \times \hdots \times \mathcal{H}(F_k) \times B \right) \\
&\leq \mathbb{P} \left( \left( A \left(X^{(1)} \right), \hdots, A \left( X^{(k)} \right), \sigma  \right)  \in  \mathcal{H} (F_1) \times \hdots \times \mathcal{H}(F_k) \times B \right).
\end{align*}
This shows
\begin{equation} \label{TempEq2ExtendedArginfTheorem}
\begin{aligned}
&\limsup_{\alpha} \mathbb{P} \left( \xi_{\alpha,j} \in F_j, 1 \le j \le k, \sigma_{\alpha} \in B \right) \\
&\leq \mathbb{P} \left( A \left( X^{(j)} \right) \cap F_j \neq \emptyset, 1 \le j \le k, \sigma \in B \right)
\end{aligned}
\end{equation}
for all compact $F_j \subseteq \mathbb{R}^{d_j}$ and closed $B \in \mathcal{B}(S)$. Next, assume that the $F_j$ are closed. Since
\begin{align*}
\Omega = \left\lbrace \Vert \xi_{\alpha} \Vert_{\infty} \leq a \right\rbrace \cup \left\lbrace \Vert \xi_{\alpha} \Vert_{\infty} > a \right\rbrace,
\end{align*}
for each $\alpha \in I$ and $a > 0$, we obtain:
\begin{equation} \label{TempEq3ExtendedArginfTheorem}
\begin{aligned}
&\limsup_{\alpha} \mathbb{P} \left( \xi_{\alpha,j} \in F_j, 1 \le j \le k, \sigma_{\alpha} \in B \right) \\
&\leq \limsup_{\alpha} \mathbb{P} \left( \xi_{\alpha,j} \in F_j, 1 \le j \le k, \sigma_{\alpha} \in B, \Vert \xi_{\alpha} \Vert_{\infty} \leq a \right) + \limsup_{\alpha} \mathbb{P} \left( \Vert \xi_{\alpha} \Vert_{\infty} > a \right) \\
&=: P_1 (a) + P_2 (a)
\end{aligned}
\end{equation}
for each $a > 0$. From
\begin{align*}
\left\lbrace \Vert \xi_{\alpha} \Vert_{\infty} \leq a \right\rbrace = \left\lbrace \xi_{\alpha,j} \in [-a,a]^{d_j}, 1 \le j \le k \right\rbrace,
\end{align*}
it follows that
\begin{align*}
P_1(a) &=\limsup_{\alpha} \mathbb{P} \left( \xi_{\alpha,j} \in F_j, 1 \le j \le k, \sigma_{\alpha} \in B, \Vert \xi_{\alpha} \Vert_{\infty} \leq a \right)\\
&= \limsup_{\alpha} \mathbb{P} \left( \xi_{\alpha,j} \in F_j \cap [-a,a]^{d_j}, 1 \le j \le k, \sigma_{\alpha} \in B \right) \\
&\leq \mathbb{P} \left( A \left( X^{(j)} \right) \cap \left( F_j \cap [-a,a]^{d_j} \right) \neq \emptyset, 1 \le j \le k, \sigma \in B \right) \\
&\leq \mathbb{P} \left( A \left( X^{(j)} \right) \cap F_j \neq \emptyset, 1 \le j \le k, \sigma \in B \right)
\end{align*}
for each $a > 0$. Here, the first inequality holds by \eqref{TempEq2ExtendedArginfTheorem}, because $F_j \cap [-a,a]^{d_j}$ is compact for each $1 \le j \le k$ and every $a > 0$. The last inequality holds, because $F_j \cap [-a,a]^{d_j} \subseteq F_j$ and $\mathcal{H}(\cdot)$ is monotone increasing with respect to $\subseteq$. Since
$$
P_2(a)= \limsup_{\alpha} \mathbb{P} \left( \Vert \xi_{\alpha} \Vert_{\infty} > a \right) \rightarrow 0, \; a \rightarrow \infty
$$
by assumption \eqref{Cond3ExtendedArginfTheorem}, we obtain \eqref{LimsupIneqExArginf} for all closed $B \subseteq S$. In general, for $\sigma$-continuity sets $B$ first use $B \subseteq \overline{B}$ and then apply \eqref{LimsupIneqExArginf} to the closed $\overline{B}$, the closure
of $B$ in (any metric space) $S$. Finally, deduce
\begin{align*}
&\mathbb{P} \left( A \left( X^{(j)} \right) \cap F_j \neq \emptyset, 1 \le j \le k, \sigma \in \overline{B} \right) \\
&\qquad \qquad \qquad = \mathbb{P} \left( A \left( X^{(j)} \right) \cap F_j \neq \emptyset, 1 \le j \le k, \sigma \in B \right)
\end{align*}
from $B \subseteq \overline{B} = \mathring{B} \cup \partial B$ and $\mathring{B} \subseteq B$, where $\mathring{B}$ denotes the interior of $B$ in (any metric space) $S$. \\

For the proof of \eqref{LiminfIneqExArginf} let firstly $K_1,\hdots, K_k$ be compact. We use the notation $M^C$ for the complement of any set $M$. Since
\begin{align*}
\left\lbrace \xi_{\alpha,j} \in K_j^{\mathrm{C}} \right\rbrace \supseteq \left\lbrace A \left(X_{\alpha}^{(j)} \right) \in \mathcal{M}(K_j) \right\rbrace,
\end{align*}
for each $1 \le j \le k$, it follows that
\begin{equation} \label{TempEq4ExtendedArginfTheorem}
\begin{aligned}
&\liminf_{\alpha} \mathbb{P} \left( \xi_{\alpha,j} \in K_j^{\mathrm{C}}, 1 \le j \le k, \sigma_{\alpha} \in B \right) \\
&= \liminf_{\alpha} \mathbb{P} \left( \left( A \left(X^{(1)}_{\alpha} \right), \hdots, A \left( X^{(k)}_{\alpha} \right), \sigma_{\alpha} \right) \in  \mathcal{M}(K_1) \times \hdots \times \mathcal{M}(K_k) \times \mathring{B} \right) \\
&\geq \mathbb{P} \left( A \left( X^{(j)} \right) \cap K_j = \emptyset, 1 \le j \le k, \sigma \in \mathring{B} \right) \\
&= \mathbb{P} \left( A \left( X^{(j)} \right) \subseteq K_j^{\mathrm{C}}, 1 \le j \le k, \sigma \in B \right).
\end{aligned}
\end{equation}
Here, the inequality $\ge$ is a consequence of (\ref{TempEq1ExtendedArginfTheorem}) and the Portmanteau-Theorem, because
$\mathcal{M}(K_1) \times \hdots \times \mathcal{M}(K_k) \times \mathring{B}$ is open in $\mathcal{F}_{d_1} \times \hdots \times \mathcal{F}_{d_k} \times S.$ The last equality holds, since $B$ is a $\sigma$-continuity-set.

Next, let $I_j := [-a,a]^{d_j}$ for each $1 \le j \le k$ and $a > 0$. Then
\begin{align*}
\left\lbrace \xi_{\alpha,j} \in G_j, \xi_{\alpha,j} \in I_j \right\rbrace = \left\lbrace \xi_{\alpha,j} \in \left( G_j^{\mathrm{C}} \cap I_j \right)^{\mathrm{C}} \cap I_j \right\rbrace,
\end{align*}
because $(G_j^{\mathrm{C}} \cap I_j)^{\mathrm{C}} \cap I_j = G_j \cap I_j$. Hence we obtain
\begin{equation} \label{TempEq5ExtendedArginfTheorem}
\begin{aligned}
&\liminf_{\alpha} \mathbb{P} \left( \xi_{\alpha,j} \in G_j, 1 \le j \le k, \sigma_{\alpha} \in B \right) \\
&\geq \liminf_{\alpha} \mathbb{P} \left( \xi_{\alpha,j} \in G_j, 1 \le j \le k, \sigma_{\alpha} \in B, \Vert \xi_{\alpha} \Vert_{\infty} \leq a \right) \\
&= \liminf_{\alpha} \mathbb{P} \left( \xi_{\alpha,j} \in K_j^{\mathrm{C}}, 1 \le j \le k, \sigma_{\alpha} \in B, \Vert \xi_{\alpha} \Vert_{\infty} \leq a \right),
\end{aligned}
\end{equation}
where $K_j := G_j^{\mathrm{C}} \cap I_j$ is compact for each $1 \le j \le k$. Since $\mathbb{P}(A \cap B) \geq \mathbb{P}(A) - \mathbb{P}(B^{\mathrm{C}})$ for all $A,B \in \mathcal{A}$, the last limit inferior in \eqref{TempEq5ExtendedArginfTheorem} has the following lower bound:
\begin{align*}
&\liminf_{\alpha} \mathbb{P} \left( \xi_{\alpha,j} \in K_j^{\mathrm{C}}, 1 \le j \le k, \sigma_{\alpha} \in B \right) - P_2 (a) \\
&\geq \mathbb{P} \left( A \left( X^{(j)} \right) \subseteq K_j^{\mathrm{C}}, 1 \le j \le k, \sigma \in B \right) - P_2 (a) \\
&\geq \mathbb{P} \left( A \left(X^{(j)} \right) \subseteq G_j, 1 \le j \le k, \sigma \in B \right) - P_2(a).
\end{align*}
Here, the first inequality holds by \eqref{TempEq4ExtendedArginfTheorem} and the last one holds, because $K_j \subseteq G_j^{\mathrm{C}}$, whence $K_j^{\mathrm{C}} \supseteq G_j$. Taking the limit $a \rightarrow \infty$ yields \eqref{LiminfIneqExArginf} upon noticing that $P_2(a) \rightarrow 0$ by \eqref{Cond3ExtendedArginfTheorem}.
\end{proof}

\begin{remark} \label{capacitycontainmentfunctional}
Notice that for general sets $E_j \subseteq \mathbb{R}^{d_j}$ the following equivalences
hold:
$$
 A(X^{(j)}) \cap E_j \neq \emptyset, 1 \le j \le k \; \Longleftrightarrow \; A(X^{(1)})\times \cdots \times A(X^{(k)}) \cap (E_1\times \cdots \times E_k) \neq \emptyset
$$
and
$$
 A(X^{(j)}) \subseteq E_j, 1 \le j \le k \; \Longleftrightarrow \; A(X^{(1)})\times \cdots \times A(X^{(k)}) \subseteq E_1\times \cdots \times E_k.
$$

Consider $C:= A(X^{(1)})\times \cdots \times A(X^{(k)})\times \{\sigma\}$. We already know that each $A(X^{(j)})$ is a random closed set in $\mathbb{R}^{d_j}, 1 \le j \le k,$ and clearly $\{\sigma\}$ is a random closed set in $S$. Thus, by Theorem 1.3.25 in Molchanov \cite{Molchanov2017}
the cartesian product $C$ is a random closed set in $H=\mathbb{R}^{d_1} \times \cdots \times \mathbb{R}^{d_k} \times S$. Let $\mu$ and $\nu$ be the capacity-functional and the containment-functional, respectively, of $C$ in $H$, i.e.
$$
 \mu(E)=\mathbb{P}(C \cap E \neq \emptyset) \; \text{ and } \; \nu(E)=\mathbb{P}(C \subseteq E) \; \text{ for all Borel-sets } E \text{ in } H.
$$
If $P_\alpha$ is the distribution of $(\xi_{\alpha,1},\hdots, \xi_{\alpha,k},\sigma_{\alpha})$, then our limit-results (\ref{LimsupIneqExArginf}) and (\ref{LiminfIneqExArginf}) can equivalently be rewritten as
$$
 \limsup_\alpha P_\alpha(F_1 \times \cdots \times F_k \times B) \le \mu(F_1 \times \cdots \times F_k \times B)
$$
and
$$
 \liminf_\alpha P_\alpha(G_1 \times \cdots \times G_k \times B) \ge \nu(G_1 \times \cdots \times G_k \times B).
$$
Therefore, we say that the $P_\alpha$ converge weakly to $(\mu,\nu)$.
\end{remark}

Assume that each limit process $X^{(j)}$ has a unique minimizing point $\xi_j$ with probability one. Then $C$ simplifies to
$C=\{\xi_1\} \times \cdots \times \{\xi_k\} \times \{\sigma\}$ and hence $\mu=\nu=P$. Thus, we may expect weak convergence
$P_\alpha \rightarrow_w P$. The following corollary shows that this is indeed the case

\begin{corollary} \label{CorollaryExtendedArginf}
Under the assumptions of Theorem \ref{ExtendedArginfTheorem} suppose in addition that there are random variables $\xi_j \in \mathbb{R}^{d_j}$ such that $A(X^{(j)}) = \lbrace \xi_j \rbrace$ a.s. for every $1 \leq j \leq k$. Then
\begin{equation} \label{EqConvDistrExtendedArginf}
\begin{aligned}
\left( \xi_{\alpha,1},\hdots, \xi_{\alpha,k}, \sigma_{\alpha} \right) \xrightarrow{\mathcal{L}} \left( \xi_1, \hdots, \xi_k, \sigma \right) \text{ in } \mathbb{R}^{d_1} \times \cdots \mathbb{R}^{d_k} \times S.
\end{aligned}
\end{equation}
\end{corollary}

\begin{proof}
Firstly, observe that
\begin{align*}
\left\lbrace A \left( X^{(j)} \right) \cap F_j \neq \emptyset \right\rbrace = \lbrace \xi_j \in F_j \rbrace \text{ and } \left\lbrace  A \left( X^{(j)} \right) \subseteq G_j \right\rbrace  = \lbrace \xi_j \in G_j \rbrace.
\end{align*}
Now, let $B_j$ be a $\xi_j$-continuity set for $1 \le j \le k$ and $B$ a $\sigma$-continuity set. Then
\begin{align}
&\mathbb{P} \left( \xi_j \in \mathring{B}_j, 1 \le j \le k, \sigma \in B \right) \notag \\
&\leq \liminf_{\alpha} \mathbb{P} \left( \xi_{\alpha,j} \in \mathring{B}_j, 1 \le j \le k, \sigma_{\alpha} \in B \right) &&\text{ by \eqref{LiminfIneqExArginf}} \notag \\
&\leq \liminf_{\alpha} \mathbb{P} \left( \xi_{\alpha,j} \in B_j, 1 \le j \le k, \sigma_{\alpha} \in B \right) \notag \\
&\leq \limsup_{\alpha} \mathbb{P} \left( \xi_{\alpha,j} \in B_j, 1 \le j \le k, \sigma_{\alpha} \in B \right) \notag \\
&\leq \limsup_{\alpha} \mathbb{P} \left( \xi_{\alpha,j} \in \overline{B}_j, 1 \le j \le k, \sigma_{\alpha} \in B \right) \notag \\
&\leq \mathbb{P} \left( \xi_{j} \in \overline{B}_j, 1 \le j \le k, \sigma \in B \right) &&\text{ by \eqref{LimsupIneqExArginf}} \notag \\
&= \mathbb{P} \left( \xi_j \in \mathring{B}_j, 1 \le j \le k, \sigma \in B \right). \label{TempEq1CorollaryExtendedArginf}
\end{align}
To see the last equality put
\begin{align*}
M_j := \left\lbrace \xi_j \in \mathring{B}_j, \sigma \in B \right\rbrace, \qquad N_j := \left\lbrace \xi_j \in \overline{B}_j, \sigma \in B \right\rbrace, \qquad 1 \le j \le k.
\end{align*}
Then $M_j \subseteq N_j$ for each $1 \le j \le k$ and consequently
\begin{align*}
M := \bigcap_{j=1}^k M_j \subseteq \bigcap_{j=1}^k N_j =: N.
\end{align*}
Thus
\begin{align*}
0 &\leq \mathbb{P}(N) - \mathbb{P}(M) = \mathbb{P}( N \setminus M) = \mathbb{P} \Big( N \cap \Big( \bigcup_{j=1}^k M_j^{\mathrm{C}} \Big) \Big) = \mathbb{P} \Big( \bigcup_{j=1}^k \Big( N \cap M_j^{\mathrm{C}} \Big) \Big) \\
&\leq \mathbb{P} \Big( \bigcup_{j = 1}^k \Big( N_j \cap M_j^{\mathrm{C}} \Big) \Big) \leq \sum_{j=1}^k \mathbb{P} \left( N_j \setminus M_j \right) = \sum_{j=1}^k \Big( \mathbb{P}(N_j) - \mathbb{P}(M_j) \Big) = 0,
\end{align*}
because every summand vanishes To see this, recall that $\overline{B}_j = \mathring{B}_j \cup \partial B_j$, whence
\begin{align*}
\mathbb{P}(M_j) &\leq \mathbb{P}(N_j) = \mathbb{P} \left( \xi_j \in \overline{B}_j, \sigma \in B \right) \\
&\leq \mathbb{P} \left( \xi_j \in \mathring{B}_j, \sigma \in B \right) + \mathbb{P} \left( \xi_j \in \partial B_j \right) \\
&= \mathbb{P}(M_j).
\end{align*}
This shows the equality \eqref{TempEq1CorollaryExtendedArginf}, and further that
\begin{align*}
\mathbb{P} \left( \xi_j \in \mathring{B}_j, 1 \le j \le k, \sigma \in B \right) = \mathbb{P} \left( \xi_j \in B_j, 1 \le j \le k, \sigma \in B \right),
\end{align*}
because $\mathring{B}_j \subseteq B_j \subseteq \overline{B}_j$. So, finally we arrive at
\begin{align*}
\lim_{\alpha} \mathbb{P} \left( \xi_{\alpha,j} \in B_j, 1 \le j \le k, \sigma_{\alpha} \in B \right) = \mathbb{P} \left( \xi_j \in B_j, 1 \le j \le k, \sigma \in B \right),
\end{align*}
and the assertion follows from Theorem 2.8 in \cite{Billingsley1999}, which -as its proof shows- also holds for nets.
\end{proof}
\noindent

Our next result involves the smallest and largest minimizer
\begin{align*}
\xi_{\text{min}}^{(j)}:= \allowbreak \text{sargmax}(X^{(j)}) \text{ and } \xi_{\text{max}}^{(j)}:=\text{largmax}(X^{(j)})
\end{align*}
of $X^{(j)}$, $1 \leq j \leq k$. If $X^{(j)}$ is \emph{coercive}, i.e. $X^{(j)}(t) \rightarrow \infty$ as $|t| \rightarrow \infty$, then $A(X^{(j)})$ is bounded,
whence as a closed set it is compact. Consequently, $\xi_{\text{min}}^{(j)} \in \mathbb{R}^{d_j}$ and $\xi_{\text{max}}^{(j)} \in \mathbb{R}^{d_j}$ in the sense of Seijo and Sen \cite{seijo2011continuous} exist.

In the following corollary we use the usual convention for vectors $x=(x^{(1)},\ldots,x^{(d)})$ and $y=(y^{(1)},\ldots,y^{(d)})$ in the euclidian space $\mathbb{R}^d: x \le y$, if $x^{(i)} \le y^{(i)}$ for all $i$ and $x<y$, if $x^{(i)} < y^{(i)}$  for all $i$. With this definition one obtains
that $(-\infty,x]=(-\infty,x^{(1)}] \times \ldots \times (-\infty,x^{(d)}]$ and $(-\infty,x)=(-\infty,x^{(1)}) \times \ldots \times (-\infty,x^{(d)}).$

\begin{corollary} \label{CorollaryUnivariateExtendedArginf}
Under the assumptions of Theorem \ref{ExtendedArginfTheorem} suppose in addition that every limit process $X^{(j)}$, $1 \le j \le k,$ has a.s. a smallest and a largest minimizing point $\xi_{\text{min}}^{(j)}:=\text{sargmax}(X^{(j)})$ and $\xi_{\text{max}}^{(j)}:= \text{largmax}(X^{(j)})$ ( as for instance, when each $X^{(j)}$ is a.s. coercive.) Then
\begin{equation} \label{le}
\limsup_{\alpha} \mathbb{P} \left( \xi_{\alpha,j} \leq x_j, 1 \le j \le k, \sigma_{\alpha} \in B \right) \leq \mathbb{P} \left( \xi_{\text{min}}^{(j)} \leq x_j, 1 \le j \le k, \sigma \in B \right),
\end{equation}
and
\begin{equation} \label{l}
\liminf_{\alpha} \mathbb{P} \left( \xi_{\alpha,j} < x_j, 1 \le j \le k, \sigma_{\alpha} \in B \right) \geq \mathbb{P} \left( \xi_{\text{max}}^{(j)} < x_j, 1 \le j \le k, \sigma \in B \right),
\end{equation}
for all $x_j \in \mathbb{R}^{d_j}$, $1 \leq j \leq k$ and for all $\sigma$-continuity sets $B$.
\end{corollary}

\begin{proof}
Let $F_j = (-\infty, x_j]$ be the closed lower orthant. Then by Lemma \ref{ximaxmin} (1) in the appendix
\begin{align*}
\left\lbrace A \left( X^{(j)} \right) \cap F_j \neq \emptyset \right\rbrace = \left\lbrace \xi_{\text{min}}^{(j)} \leq x_j \right\rbrace
\end{align*}
for every $1 \le j \le k$ and (\ref{le}) follows from \eqref{LimsupIneqExArginf} in Theorem \ref{ExtendedArginfTheorem}. \\
For the proof of \eqref{l} let $G_j = (-\infty,x_j)$ be the open lower orthant. Since
\begin{align*}
\left\lbrace A \left( X^{(j)} \right) \subseteq G_j \right\rbrace = \left\lbrace \xi_{\text{max}}^{(j)} < x_j \right\rbrace
\end{align*}
for every $1 \le j \le k$ by Lemma \ref{ximaxmin} (2), the assertion follows from \eqref{LiminfIneqExArginf} in Theorem \ref{ExtendedArginfTheorem}.
\end{proof}

In case that all processes are univariate ($d_1=\cdots=d_k=1$) Corollary \ref{CorollaryUnivariateExtendedArginf} yields Theorem 2.23 in Rosar (2025) \cite{Rosar2025} (and actually slightly improves it.)

\begin{remark} \label{withoutsigma}
If one is interested only in the vector $(\xi_{\alpha,1},\hdots,\xi_{\alpha,k})$ of all minimizers, then in condition \eqref{Cond2PropositionCriteriaForJointWeakConvInProdSpaceForD(Rd)} and \eqref{Cond2*PropositionCriteriaForJointWeakConvInProdSpaceForD(Rd)} $\sigma_{\alpha}$ and $\sigma$ simply can be omitted resulting in weaker requirements. More precisely, \eqref{Cond2PropositionCriteriaForJointWeakConvInProdSpaceForD(Rd)} reduces to
 \begin{align*}
&\left( \pi_{T_1} \left( X_{\alpha}^{(1)} \right),\hdots, \pi_{T_k} \left( X_{\alpha}^{(k)} \right) \right)  \xrightarrow{\mathcal{L}} \left( \pi_{T_1} \left( X^{(1)} \right), \hdots, \pi_{T_k} \left( X^{(k)} \right) \right)
\end{align*}
in  $\mathbb{R}^{\vert T_1 \vert} \times \hdots \times \mathbb{R}^{\vert T_k \vert}$ for all finite $T_j \subseteq T_{X^{(j)}}$, $1 \le j \le k$,

and \eqref{Cond2*PropositionCriteriaForJointWeakConvInProdSpaceForD(Rd)} to
\begin{align*}
\left( \pi_T \left(X_{\alpha}^{(1)} \right), \hdots, \pi_T \left( X_{\alpha}^{(k)} \right) \right) \xrightarrow{\mathcal{L}} \left(\pi_T \left(X^{(1)} \right), \hdots, \pi_T \left( X^{(k)} \right) \right)
\end{align*}
in $\mathbb{R}^{k \vert T \vert}$ for all finite $T \subseteq T_{X^{(1)}} \cap \hdots \cap T_{X^{(k)}}$.

Then all our results correspondingly hold without the adjunction of $\sigma_{\alpha}$ and $\sigma$. For instance \eqref{LimsupIneqExArginf} becomes:
\begin{align*}
\limsup_{\alpha} \mathbb{P} \left( \xi_{\alpha,j} \in F_j, 1 \le j \le k \right) \leq \mathbb{P} \left( A \left( X^{(j)} \right) \cap F_j \neq \emptyset, 1 \le j \le k \right)
\end{align*}
for all closed $F_j \subseteq \mathbb{R}^{d_j}$, or \eqref{EqConvDistrExtendedArginf} becomes:
\begin{align*}
\left( \xi_{\alpha,1}, \hdots, \xi_{\alpha,k} \right) \xrightarrow{\mathcal{L}} \left( \xi_1, \hdots, \xi_k \right) \text{ in } \mathbb{R}^d.
\end{align*}
\end{remark}

\section{Application} \label{SectionApplication}

We consider the following regression setting: Let $X$ and $Y$ be two real-valued random variables and suppose that there exists a function $m: \mathbb{R} \rightarrow \mathbb{R}$ such that $Y = m(X) + \epsilon$, where $\mathbb{E}[Y^2] < \infty$ and $\epsilon$ denotes a centered stochastic error. $\epsilon$ is not necessarily independent of $X$. However, it is assumed to satisfy the weaker condition $\mathbb{E}[\epsilon \vert X] = 0$ a.s. In consequence, $m(X) = \mathbb{E}[Y \vert X]$ a.s. and
\begin{equation} \label{EqRegMinL2Distance}
\begin{aligned}
m \in \argmin \left\lbrace \mathbb{E} \left[ (Y - g(X))^2 \right]: g \in L_2(\mathbb{R},\mathcal{B}(\mathbb{R}),Q)  \right\rbrace,
\end{aligned}
\end{equation}
where $Q$ is the distribution of $X$. Following the approach in Bühlmann and Yu (2002) \cite{Buehlmann2002} as well as Banerjee and McKeague (2007) \cite{Banerjee2007}, we do not  estimate $m$ itself. Rather we want to estimate the minimizer of the $L_2$-distance in \eqref{EqRegMinL2Distance} on a smaller domain than $L_2(\mathbb{R},\mathcal{B}(\mathbb{R}),Q)$. For this purpose, let $k \in \mathbb{N}$ and
\begin{align*}
\Lambda_{<}^k := \left\lbrace (t_1,\hdots,t_k) \in \mathbb{R}^k : t_1 < \hdots < t_k \right\rbrace.
\end{align*}
Each tuple of parameters $(t,a) = (t_1,\hdots,t_k,a_1, \hdots,a_{k+1}) \in \Lambda_{<}^k \times \mathbb{R}^{k+1}$ uniquely determines a step function $g_{(t,a)}: \mathbb{R} \rightarrow \mathbb{R}$  with $k$ jumps by
\begin{align*}
g_{(t,a)}(x) := a_1 \1_{x \leq t_1} + a_2 \1_{t_1 < x \leq t_2} + \hdots + a_{k} \1_{t_{k-1} < x \leq t_{k}} + a_{k+1} \1_{x > t_k}.
\end{align*}
Suppose there exists exactly one step function $g_{(\tau,\alpha)}$, that minimizes the $L_2$-distance in \eqref{EqRegMinL2Distance} among all step functions with $k$ jumps, i.e.
\begin{equation} \label{UniqueMinPropTauAlpha}
\begin{aligned}
(\tau, \alpha) \in \argmin \left\lbrace S(t,a):=  \mathbb{E} \left[ (Y - g_{(t,a)}(X))^2  \right]: (t,a) \in \Lambda_<^k \times \mathbb{R}^{k+1} \right\rbrace.
\end{aligned}
\end{equation}
Notice that $$\mathbb{E} \left[ (Y - g_{(t,a)}(X))^2  \right]=\mathbb{E} \left[ (m(X) - g_{(t,a)}(X))^2  \right]+\mathbb{E} \left[\epsilon^2  \right],$$
whence $g_{(\tau,\alpha)}$ is the best-approximation of $m$ in $L_2(\mathbb{R},\mathcal{B}(\mathbb{R}),Q)$ among all $k$-step functions.

The estimation of $(\tau,\alpha)=(\tau_1,\ldots,\tau_k,\alpha_1,\ldots,\alpha_{k+1})$ is a parametric problem, even if the regression function $m$ is non-parametric. It is assumed that $m$ satisfies the assumptions of Chapter 4 in Rosar \cite{Rosar2025}, where $m$ itself has discontinuities in the jumps of its best $L_2$-approximation $g_{(\tau,\alpha)}$. In detail we suppose for each $1 \le j \le k$ that $m$ is continuous in an open neighbourhood of $\tau_j$ with the exception of $m(\tau_j -) \neq m(\tau_j +)$. \\
\, \\
\begin{figure}[ht]
\centering
\begin{subfigure} {0.48\textwidth}
\includegraphics{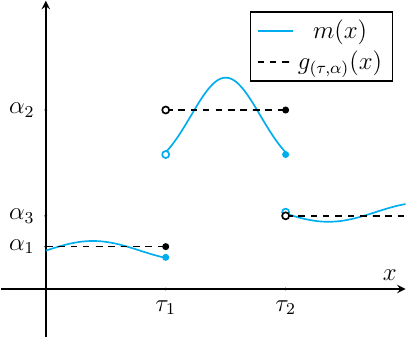}
\caption{Exemplary sketch of a regression function and its best $L_2$-approximation for $k=2$.}
\label{fig:Fig1}
\end{subfigure} \hfill
\begin{subfigure} {0.48\textwidth}
\centering
\includegraphics{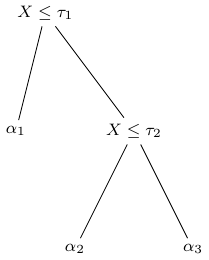}
\caption{Decision tree that results from minimizing parameters $(\tau, \alpha)$ in \eqref{UniqueMinPropTauAlpha} for $k=2$.}
\label{fig:Fig2}
\end{subfigure}
\end{figure}
\, \\
Note that the parameter $(\tau,\alpha)$ can be interpreted as a decision tree in the tradition of Breiman (1993) \cite{Breiman1993}.
To estimate $(\tau, \alpha)$ we use the least squares method: For each $n \in \mathbb{N}$ let $(X_i,Y_i)$, $1 \leq i \leq n$, be i.i.d copies of $(X,Y)$. Then we define an estimator $(\tau_n,\alpha_n)$ as (any) minimizer of the empirical counterpart $S_n$ of $S$,
$$
 S_n(t,a):=\frac{1}{n} \sum_{i=1}^{n} \left( Y_i - g_{(t,a)}(X_i) \right)^2, \text{ i.e. }
$$
\begin{equation} \label{DefEstimator}
\begin{aligned}
(\tau_n, \alpha_n)=(\tau_{n,1},\ldots,\tau_{n,k},\alpha_{n,1},\ldots,\alpha_{n,k+1}) \in A(S_n).
\end{aligned}
\end{equation}
Theorem \ref{ExtendedArginfTheorem} enables us to derive a limit theorem including convergence rates for $(\tau_n, \alpha_n)$ in (\ref{DefEstimator}). We can prove the convergence of
\begin{align*}
\begin{pmatrix}
n (\tau_n - \tau) \\
\sqrt{n} (\alpha_n - \alpha)
\end{pmatrix}
\end{align*}
and identify the limit variable. To do this, for each $n \in \mathbb{N}$ we introduce the \emph{rescaled processes} $Z_n: \mathbb{R}^k \rightarrow \mathbb{R}$ in $D(\mathbb{R}^k)$ in the following way:
\begin{align*}
Z_n(t) = n \left\lbrace S_n(\tau + n^{-1} t, \alpha_n) - S_n(\tau,\alpha_n) \right\rbrace,
\end{align*}
and see from \eqref{DefEstimator} that $n (\tau_n - \tau) \in A(Z_n)$. Rosar \cite{Rosar2025} decomposes for each $n \in \mathbb{N}$ the multivariate $Z_n$ into the sum of $k$ univariate processes $Z_n^{(j)}$  in $D(\mathbb{R}), 1 \le j \le k$, where
\begin{equation} \label{MinRescaledProcesses}
\begin{aligned}
n (\tau_{n,j} - \tau_j) \in A(Z_n^{(j)}).
\end{aligned}
\end{equation}
By Lemma 6.10 in Albrecht \cite{albrecht2020} it turns out that for each $1 \leq j \leq k$ the processes $(Z_n^{(j)})_{n \in \mathbb{N}}$ converge weakly to a certain compound Poisson process $Z^{(j)}$ in $D(\mathbb{R})$. The exact form of these $Z^{(j)}$ is given in Definition 4.10 in \cite{Rosar2025}. This means that \eqref{Cond1PropositionCriteriaForJointWeakConvInProdSpaceForD(Rd)} in Theorem \ref{ExtendedArginfTheorem} is fulfilled. \\
Moreover, $\sqrt{n} (\alpha_n - \alpha)$ converges in distribution to a centered normal vector $W = (W_1,\hdots,W_{k+1})$, defined in Theorem 4.20 in Rosar \cite{Rosar2025}. Here,
\begin{equation} \label{allareindependent}
Z^{(1)}, \hdots, Z^{(k)}, W_1, \hdots W_{k+1} \; \text{ are independent.}
\end{equation}
 Among other things, this follows from Lemma 4.22 in Rosar \cite{Rosar2025}, which says that
\begin{equation}
\begin{aligned} \label{ConvFidisZn}
&\Big( \pi_T \left( Z_n^{(1)} \right), \hdots, \pi_T \left( Z_n^{(k)} \right), \sqrt{n} (\alpha_{n} - \alpha) \Big) \xrightarrow{\mathcal{L}} \left( \pi_T( Z^{(1)}),\hdots, \pi_T(Z^{(k)}), W \right)
\end{aligned}
\end{equation}
in $\mathbb{R}^{k \cdot \vert T \vert} \times \mathbb{R}^{k+1}$ for all finite $T \subseteq \mathbb{R}$.
So, we see that condition \eqref{Cond2*PropositionCriteriaForJointWeakConvInProdSpaceForD(Rd)} is also satisfied.
As to \eqref{Cond3ExtendedArginfTheorem}, the stochastic boundedness of $(n (\tau_n - \tau))_{n \in \mathbb{N}}$ holds by Theorem 4.16 in Rosar \cite{Rosar2025}, i.e.
\begin{equation} \label{StochBoundednessTaun}
\begin{aligned}
\lim_{a \rightarrow \infty} \limsup_{n \rightarrow \infty} \mathbb{P} \left( \left\Vert n (\tau_n - \tau) \right\Vert_{\infty} > a \right) = 0.
\end{aligned}
\end{equation}
In conclusion, we can apply Theorem \ref{ExtendedArginfTheorem}. This leads to Theorem \ref{LimitTheoremRegrEstimator} below.
For its formulation, let $\mu_j$ and $\nu_j$ be the capacity functional and the containment-functional, respectively, of the random closed set $A(Z^{(j)})$, i.e. $\mu_j(E)= \mathbb{P}(A(Z^{(j)}) \cap E \neq \emptyset)$ and $\nu_j(E)=\mathbb{P}(A(Z^{(j)}) \subseteq E)$ for every Borel-set $E \subseteq \mathbb{R}$. Recall that $\mu_j$ in particularly is a Choquet-capacity and that $\nu_j(E)=1-\mu_j(E^C)$.

\begin{theorem} \label{LimitTheoremRegrEstimator}
Let the assumptions 2.1, 2.2, 2.4, 4.1 and 4.2 in Rosar \cite{Rosar2025} be valid. Further, suppose that
\begin{align*}
(\tau_n, \alpha_n) \xrightarrow{\mathbb{P}} (\tau, \alpha), \qquad n \rightarrow \infty.
\end{align*}
Then
\begin{equation} \label{limsup}
\begin{aligned}
&\limsup_{n \rightarrow \infty} \mathbb{P} \Big( n (\tau_{n,j}-\tau_j) \in F_j, 1\le j \le k,  n^{\frac{1}{2}} (\alpha_{n,i}- \alpha_i) \in B_i, 1 \le i\le k+1)\Big)\\
&\le \prod_{j=1}^k \mu_j(F_j) \prod_{i=1}^{k+1} \mathbb{P}( W_i \in B_i ),
\end{aligned}
\end{equation}
for all closed sets $F_j \subseteq \mathbb{R}$ and for all $W_i$-continuity sets $B_i \subseteq \mathbb{R}$, and
\begin{equation} \label{liminf}
\begin{aligned}
&\liminf_{n \rightarrow \infty} \mathbb{P} \Big( n (\tau_{n,j}-\tau_j) \in G_j, 1\le j \le k,  n^{\frac{1}{2}} (\alpha_{n,i}- \alpha_i) \in B_i, 1 \le i\le k+1)\Big)\\
&\ge \prod_{j=1}^k \nu_j(G_j) \prod_{i=1}^{k+1} \mathbb{P}( W_i \in B_i ),
\end{aligned}
\end{equation}
for all open sets $G_j \subseteq \mathbb{R}$ and for all $W_i$-continuity sets $B_i \subseteq \mathbb{R}$.
\end{theorem}

\begin{proof} Recall that we can apply Theorem \ref{ExtendedArginfTheorem}. Thus, by (\ref{LimsupIneqExArginf}) the upper limit in (\ref{limsup}) is less than or equal to
\begin{equation}
\begin{aligned}
&\mathbb{P}(A(Z^{(j)}) \cap F_j \neq \emptyset, 1 \le j \le k, W_i \in B_i, 1 \le i \le k+1)\\
&= \prod_{j=1}^k \mathbb{P}(A(Z^{(j)}) \cap F_j \neq \emptyset) \prod_{i=1}^{k+1} \mathbb{P}(W_i \in B_i),
\end{aligned}
\end{equation}
where the last equality holds by (\ref{allareindependent}). Thus, (\ref{limsup}) follows from the definition of the capacity functionals $\mu_j$.
Similarly, the second part of Theorem \ref{ExtendedArginfTheorem} says that the lower limit in (\ref{liminf}) is greater than or equal
\begin{equation}
\begin{aligned}
&\mathbb{P}(A(Z^{(j)}) \subseteq G_j \neq \emptyset, 1 \le j \le k, W_i \in B_i, 1 \le i \le k+1)\\
&= \prod_{j=1}^k \mathbb{P}(A(Z^{(j)}) \subseteq G_j \neq \emptyset) \prod_{i=1}^{k+1} \mathbb{P}(W_i \in B_i).
\end{aligned}
\end{equation}
Now, (\ref{liminf}) follows immediately from the definition of the con\-tainment-func\-tionals $\nu_j$.
\end{proof}

\noindent
This convergence result can be used for the construction of confidence rectangles. Here, first of all it should be noted that $Z^{(j)}(t) \rightarrow \infty$ a.s as $|t| \rightarrow \infty$. Therefore, $A(Z^{(j)})$ is compact as a union of at most finitely many compact intervals. Thus,
its smallest and largest minimizing point $\xi_{\text{min}}^{(j)}$ and $\xi_{\text{max}}^{(j)}$ a.s. exist and are real random variables.

\begin{corollary} \label{ConfIntervalsRegrEstimator}
Let the assumptions of Theorem \ref{LimitTheoremRegrEstimator} be valid and $\rho \in (0,1)$. Further, define $\gamma = (1 - \rho)^{\frac{1}{2k+1}}$ and choose
\begin{itemize}
\item for each $1 \le j \le k$: $a_j, b_j \in \mathbb{R}$ such that
\begin{align*}
\mathbb{P} \left( \xi_{\text{min}}^{(j)} > a_j, \xi_{\text{max}}^{(j)} < b_j \right) \ge \gamma,
\end{align*}
\item for each $i \in \lbrace 1,\hdots,k+1 \rbrace$: $v_i$ as the $\frac{\gamma+1}{2}$-quantile and $u_i$ as the $\frac{1-\gamma}{2}$-quantile of the normal distribution function $\Phi_{(0,\sigma_i)}$ with zero mean and variance $\sigma_i^2$, where the variances are given in Theorem 4.20 in \cite{Rosar2025}.
\end{itemize}
Then, the $(2 k+1)$-dimensional rectangle
\begin{equation*}
I_n := \prod_{j=1}^k \left( \tau_{n,j} - n^{-1} b_j, \tau_{n,j} - n^{-1} a_j \right) \times \prod_{i=1}^{k+1} \left[ \alpha_{n,i} - n^{-\frac{1}{2}} v_i, \alpha_{n,i} - n^{-\frac{1}{2}} u_i \right]
\end{equation*}
is an asymptotic confidence region for $(\tau,\alpha) \in \mathbb{R}^{2k+1}$ at level $(1-\rho)$, that is
\begin{equation} \label{asymptoticcoverageprobability}
 \liminf_{n \rightarrow \infty} \mathbb{P} \left( (\tau, \alpha) \in I_n \right) \ge 1-\rho.
\end{equation}
\end{corollary}

\begin{proof} First, notice that $\gamma \in (0,1)$ and so $\frac{1+\gamma}{2} \in (\frac{1}{2},1)$ and $\frac{1-\gamma}{2} \in (0,\frac{1}{2})$, whence the quantiles $v_i$ and $u_i$ are well-defined.
Choose $G_j = (a_j, b_j)$, $1 \le j \le k$, and $B = [u_1, v_1] \times \hdots \times [u_{k+1},v_{k+1}]$. Then it follows from (\ref{liminf}) in Theorem \ref{LimitTheoremRegrEstimator} that
\begin{align*}
&\liminf_{n \rightarrow \infty} \mathbb{P} \left( (\tau, \alpha) \in I_n \right) \\
&= \liminf_{n \rightarrow \infty} \mathbb{P} \big( n (\tau_{n,j} - \tau_j) \in (a_j,b_j), 1\le j \le k, \\
&\qquad \qquad \qquad \qquad  u_i \le \sqrt{n} (\alpha_{n,i} - \alpha_)\le v_i, 1 \le i \le k+1 \big) \\
&\geq \prod_{j=1}^k \mathbb{P} \left( A \left( Z^{j} \right) \subseteq (a_j,b_j) \right) \cdot \prod_{j=1}^{k+1} \left( \Phi_{(0,\sigma_j)}(v_j) - \Phi_{(0,\sigma_j)}(u_j) \right) \\
&= \prod_{j=1}^k \mathbb{P} \left( \xi_{\text{min}}^{(j)} > a_j, \xi_{\text{max}}^{(j)} < b_j \right) \cdot \prod_{j=1}^{k+1} \left( \Phi_{(0,\sigma_j)}(v_j) - \Phi_{(0,\sigma_j)}(u_j) \right),
\end{align*}
Therefore, by definition of $a_j, b_j, u_j$ and $v_j$ the last product in the above display is greater than or equal to $\gamma^{2k+1}=1-\rho$, which gives the desired result (\ref{asymptoticcoverageprobability}).
\end{proof}

Notice that $d_1=\ldots=d_k=1$ and that by (\ref{allareindependent})
$$
\xi_{\text{min}}^{(1)}, \ldots, \xi_{\text{min}}^{(k)}, W_1,\ldots, W_{k+1} \text{ are independent}
$$
and
$$
\xi_{\text{max}}^{(1)}, \ldots, \xi_{\text{max}}^{(k)}, W_1,\ldots, W_{k+1} \text{ are independent}.
$$
Thus, an application of Corollary \ref{CorollaryUnivariateExtendedArginf} yields:

\begin{corollary} Under the assumptions of Theorem \ref{LimitTheoremRegrEstimator} we have that:
\begin{equation} \label{EqLimsupUnivExtendedArginfSmallInf}
\begin{aligned}
&\limsup_{\alpha} \mathbb{P} \left(  n (\tau_{n,j} - \tau_j) \leq x_j, 1 \le j \le k, \sqrt{n} (\alpha_{n,i} - \alpha_i) \in B_i, 1 \le i \le k+1\right)\\
&\leq \prod_{j=1}^k \mathbb{P} \left( \xi_{\text{min}}^{(j)} \leq x_j\right) \prod_{i=1}^{k+1} \mathbb{P}(W_i \in B_i),
\end{aligned}
\end{equation}
and
\begin{equation} \label{EqLiminfUnivExtendedArginfLargInf}
\begin{aligned}
&\liminf_{\alpha} \mathbb{P} \left(  n (\tau_{n,j} - \tau_j) < x_j, 1 \le j \le k, \sqrt{n} (\alpha_{n,i} - \alpha_i) \in B_i, 1 \le i \le k+1\right)\\
&\geq \prod_{j=1}^k \mathbb{P} \left( \xi_{\text{max}}^{(j)} < x_j\right) \prod_{i=1}^{k+1} \mathbb{P}(W_i \in B_i)
\end{aligned}
\end{equation}
for all $x_1, \hdots, x_k \in \mathbb{R}$ and for all $W_i$-continuity sets $B_i$.

\end{corollary}

\section{Declaration of competing interest}

\noindent
\textbf{Ethical Standards:} This article does not contain any studies with human participants or animals performed by the authors. \\
\, \\
\textbf{Competing Interests:} This research did not receive any specific grant from funding agencies in the public, commercial, or not-for-profit sectors. The authors declare that they have no known competing financial interests or personal relationships that could have appeared to influence the work reported in this paper.

\appendix
\section{}\label{secA1}
In this section we prove some technical results used in our proofs above. The first one involves the functionals sargmin and largmin defined in
Definition 2.4 of Seijo and Sen \cite{seijo2011continuous}, where one has to replace $f$ by $-f$.
\begin{lemma} \label{ximaxmin} Let $f \in D(\mathbb{R}^d)$ with compact $A(f)$ (as for instance, when $f$ is coercive).
For $x =(x^{(1)},\ldots,x^{(d)}) \in \mathbb{R}^d$ consider $F=(-\infty,x]$ and $G=(-\infty,x)$. Then $x_{\text{min}}:=\text{sargmin}(f)$ and $x_{\text{max}} := \text{largmin}(f)$ exist and satisfy the following two implications:
\begin{itemize}
\item[(1)] $A(f) \cap F \neq \emptyset \quad \Leftrightarrow \quad x_{\text{min}} \le x.$
\item[(2)] $x_{\text{max}} < x \quad \Leftrightarrow \quad A(f) \subseteq G.$
\end{itemize}
\end{lemma}

\begin{proof} (1) The direction $\Leftarrow$ is trivial. For the reverse one observe that, if $A(X) \cap F \neq \emptyset$, then there exists a point $y=(y^{(1)},\ldots,y^{(d)}) \in A(X) \cap F$.
Since in particularly $y$ is a minimizing point of $f$, it follows from the definition of $x_{\text{min}}=\mbox{sargmin}(f)$ that $x_{\text{min}}^{(1)} \le y^{(1)}$.
Moreover, $y^{(1)} \le x^{(1)}$, because $y \in F$. Consequently, $x_{\text{min}}^{(1)} \le x^{(1)}$ and therefore $(x_{\text{min}}^{(1)}, y^{(2)},\ldots,y^{(d)}) \in A(X) \cap F$. Applying the definition of sargmin$(f)$ again, we obtain that $x_{\text{min}}^{(2)} \le y^{(2)} \le x^{(2)}$ and so
$(x_{\text{min}}^{(1)}, x_{\text{min}}^{(2)}, y^{(3)} \ldots,y^{(d)}) \in A(X) \cap F$. Proceeding successively in that way one arrives at
$(x_{\text{min}}^{(1)}, x_{\text{min}}^{(2)}, \ldots,x_{\text{min}}^{(d)}) \in A(X) \cap F \subseteq F$. This shows the first implication (1).

As to (2) firstly notice that $\Leftarrow$ is obvious.  Thus, assume that $x_{\text{max}} < x $ and let $y=(y^{(1)},\allowbreak \ldots,y^{(d)}) \allowbreak \in A(f)$. Then $y^{(1)} \le x_{\text{max}}^{(1)}$ be definition of $x_{\text{max}}$ as largmax$(f)$. Consequently, $y^{(1)} < x^{(1)}$. Moreover, $(x_{\text{max}}^{(1)},y^{(2)},\ldots,y^{(d)}) \in A(f)$. Another application of the definition
of largmax$(f)$ yields that $y^{(2)} \le x_{\text{max}}^{(2)}< x^{(2)}$ and  that $(x_{\text{max}}^{(1)},x_{\text{max}}^{(2)}, y^{(3)},\ldots,y^{(d)}) \in A(f)$. Thus, we successively get that $y < x$, whence $y \in G$ as desired.
\end{proof}

\begin{lemma} \label{Lemma1Appendix}
Let $X_j$, $j \in \mathbb{N}$, be random variables in $D(\mathbb{R}^d)$. Then $T := \bigcap_{j \in \mathbb{N}} T_{X_j}$ lies dense in $\mathbb{R}^d$.
\end{lemma}

\begin{proof}
Assume that the set $T$ is not dense in $\mathbb{R}^d$. Then there exist a point $x \in \mathbb{R}^d$ and a positive real $\epsilon$ such that the ball $B(x,\epsilon)$ with center $x$ and radius $\epsilon$ does not contain at least one element of $T$, i.e.
\begin{align*}
B(x,\epsilon) \subseteq T^{\mathrm{C}} = \bigcup_{j \in \mathbb{N}} T_{X_j}^{\mathrm{C}}.
\end{align*}
Let $a > 0$ with $B(x,\epsilon) \subseteq [-a,a]^d =: I_a$. Then
\begin{align} \label{BsubsetTX}
B(x,\epsilon) \subseteq \bigcup_{j \in \mathbb{N}} \left( T_{X_j}^{\mathrm{C}} \cap I_a \right).
\end{align}
We know that for each $j \in \mathbb{N}$,
\begin{align*}
T_{X_j}^{\mathrm{C}} = \left\lbrace t \in \mathbb{R}^d : X_j \text{ is not continuous at } t \text{ with positive probability} \right\rbrace.
\end{align*}
Let $X_j^{(a)}$ denote the restriction of $X_j$ to $I_a$. Then $X_j^{(a)}$ is a random variable in $D([-a,a]^d)$ introduced and studied by \cite{Neuhaus1971}. Now,
\begin{equation} \label{TXsubsetH}
\begin{aligned}
&T_{X_j}^{\mathrm{C}} \cap I_a \\
&= \left\lbrace t \in [-a,a]^d : X_j^{(a)} \text{ is discontinuous at $t$ with positive probability} \right\rbrace \\
&\subseteq H_j,
\end{aligned}
\end{equation}
where $H_j$ is the countable union of proper hyperplanes in $I_a$, confer Neuhaus (1971) \cite{Neuhaus1971}, p. 1290. Since the Lebesgue-measure $\lambda$ of every proper hyperplane is equal to zero, $\lambda(H_j) = 0$ for each $j \in \mathbb{N}$ by countability. Deduce from (\ref{BsubsetTX}) and (\ref{TXsubsetH}) that
\begin{align*}
B(x,\epsilon) \subseteq \bigcup_{j \geq 1} H_j,
\end{align*}
whence
\begin{align*}
0 \leq \lambda ( B(x,\epsilon) ) \leq \sum_{j \geq 1} \lambda(H_j) = 0,
\end{align*}
a contradiction.
\end{proof}

\bibliographystyle{elsarticle-num-names}
\bibliography{references}

\end{document}